\documentclass[sn-mathphys,Numbered]{sn-jnl}

\usepackage{graphicx}%
\usepackage{multirow}%
\usepackage{amsmath,amssymb,amsfonts}%
\usepackage{amsthm}%
\usepackage{mathrsfs}%
\usepackage[title]{appendix}%
\usepackage{xcolor}%
\usepackage{textcomp}%
\usepackage{manyfoot}%
\usepackage{booktabs}%
\usepackage{algorithm}%
\usepackage{algorithmicx}%
\usepackage{algpseudocode}%
\usepackage{listings}%

\usepackage{siunitx}
\usepackage[gen]{eurosym}
\usepackage[caption=false]{subfig}
\usepackage{filecontents}

\usepackage{nomencl}
\makenomenclature
\usepackage{etoolbox}
\renewcommand\nomgroup[1]{%
  \item[\bfseries
  \ifstrequal{#1}{A}{Sets and Indices}{%
  \ifstrequal{#1}{B}{Parameters}{%
  \ifstrequal{#1}{D}{Dual Variables}{%
  \ifstrequal{#1}{C}{Variables}{}}}}%
]}

\renewcommand{\nompreamble}{The next list describes several symbols that are used within the body of the document.}

\usepackage{tikz, tikzscale}
\usepackage{pgfplots}
\usepgfplotslibrary{fillbetween}
\pgfplotsset{compat=1.7}

\usetikzlibrary{patterns}
\usetikzlibrary{datavisualization.formats.functions}
\usetikzlibrary{decorations.pathreplacing,calligraphy}

\definecolor{julia1}{rgb}{0.0,0.6056031611752245,0.9786801175696073}
\definecolor{julia2}{rgb}{0.8888735002725198,0.43564919034818994,0.2781229361419438}
\definecolor{julia3}{rgb}{0.2422242978521988,0.6432750931576305,0.3044486515341153}
\definecolor{julia4}{rgb}{0.7644401754934356,0.4441117794687767,0.8242975359232758}
\definecolor{julia5}{rgb}{0.6755439572114057,0.5556623322045815,0.09423433626639477}
\definecolor{julia6}{rgb}{4.821181644776295e-7,0.6657589812923561,0.6809969518707945}
\definecolor{julia7}{rgb}{0.930767491919665,0.3674771896571412,0.5757699667547829}

\definecolor{maincolor}{HTML}{032F99}
\definecolor{BrickRed}{HTML}{cb4154}
\definecolor{color_ConvD}{rgb}{0.8588,0.3333,0.3333}
\definecolor{color_darkgreen}{rgb}{0.10,0.30,0.00}
\definecolor{color_darkblue}{rgb}{0.00,0.25,0.50}
\definecolor{color_darkred}{rgb}{0.80,0.13,0.00}

\theoremstyle{thmstyleone}%
\newtheorem{theorem}{Theorem}
\newtheorem{lemma}{Lemma}

\theoremstyle{thmstyletwo}%

\theoremstyle{thmstylethree}%
\newtheorem{definition}{Definition}%

\begin{document}

\title[On the Efficiency of Energy Markets with Non-Merchant Storage]{On the Efficiency of Energy Markets with Non-Merchant Storage}


\author[1]{\fnm{Linde} \sur{Fr\"{o}lke}}\email{linde.frolke@tutanota.com}
\equalcont{These authors contributed equally to this work.}

\author*[1]{\fnm{Eléa} \sur{Prat}}\email{emapr@dtu.dk}
\equalcont{These authors contributed equally to this work.}

\author[2,1]{\fnm{Pierre} \sur{Pinson}}\email{p.pinson@imperial.ac.uk}

\author[1]{\fnm{Richard M.} \sur{Lusby}}\email{rmlu@dtu.dk}

\author[1]{\fnm{Jalal} \sur{Kazempour}}\email{jalal@dtu.dk}

\affil*[1]{\orgname{Technical University of Denmark}, \orgaddress{\city{Kgs. Lyngby}, \postcode{2800}, \country{Denmark}}}
\affil[2]{\orgname{Imperial College London}, \orgaddress{\city{London}, \postcode{SW7 2DB}, \country{U.K.}}}



\abstract{Energy market designs with non-merchant storage have been proposed in recent years, with the aim of achieving optimal market integration of storage. 
In order to handle the time-linking constraints that are introduced in such markets, existing works commonly make simplifying assumptions about the end-of-horizon storage level, e.g., by imposing an exogenous level for the amount of energy to be left for the next time horizon.
This work analyzes market properties under such assumptions, as well as in their absence.
We find that, although they ensure cost recovery for all market participants, these assumptions generally lead to market inefficiencies.
Therefore we consider the design of markets with non-merchant storage without such simplifying assumptions. 
Using illustrative examples, as well as detailed proofs, we provide conditions under which market prices in subsequent market horizons fail to reflect the value of stored energy. 
We show that this problem is essential to address in order to preserve market efficiency and cost recovery.
Finally, we propose a method for restoring these market properties in a perfect-foresight setting.}




\maketitle

\section{Introduction}\label{sec:1}


The need for large-scale energy storage to balance intermittent and stochastic renewables in future energy systems has become apparent. The IEA Energy Outlook predicts that utility-scale battery storage will increase from less than 20 GW in 2020 to over 3 000 GW by 2050 \cite{iea2021world}. 
It remains an open question how to best integrate storage in energy markets.
The urgency of this question is illustrated by a recent order from the Federal Energy Regulatory Commission that requires system operators to facilitate market participation of electric storage, and to provide fair compensation for the provided services, given the physical and operational characteristics of these assets \cite{FERC2018}. 

In \cite{Singhal2020Pricing}, three possible market designs for storage systems in electricity markets are identified. They differ in their representation of storage in the objective function and constraints of the market clearing.
In the first design, there is no separate bidding format for storage systems, so that storage operators must submit price-quantity bids in the same way as loads and generators. This approach is common in current energy markets.
The second market design additionally includes storage state-of-energy constraints to the market clearing, ensuring feasibility of the storage dispatch.
In the corresponding bidding format, storage owners must therefore submit relevant operational parameters.
We use the term \textit{merchant storage} for a storage submitting price bids in the market, which is the case in these two market designs. 
One major disadvantage of merchant storage is that it can have negative effects on social welfare \cite{Hartwig2016Impact, Sioshansi2014When}, depending on its market power\footnote{Strategic bidding of storage can reduce or increase social welfare, depending on operational circumstances. Negative effects are shown in \cite{Hartwig2016Impact, Sioshansi2014When}, while \cite{Gu2022Market} shows positive effects in an oligopolistic market.}.

To mitigate these effects, the concept of \textit{non-merchant} or \textit{passive} storage in energy markets has been investigated in recent years \cite{Taylor2014Financial,Munoz2017Financial, weibelzahl2018effects,jiang2023duality}. It is the third market design option in \cite{Singhal2020Pricing}. 
A non-merchant storage does not submit price bids, but its operation is co-optimized with generation and loads to achieve the highest social welfare.
An appealing property of this set-up is that it is comparable to network-aware markets, which are prevalent in practice, e.g., in the U.S. markets. Similar to the way that power lines are included in the market to perform spatial arbitrage, a non-merchant storage performs temporal arbitrage. Thus, the storage is scheduled to achieve more efficient use of energy systems.
Non-merchant storage can be a public asset owned by the system or network operator, but it could also be privately owned. 

Time-linking constraints are inevitably part of optimal dispatch problems for storage.
With merchant storage, these time-linking constraints can be considered in the individual optimal bidding problem of the storage operator. The challenges associated with time-linking constraints in this case have been studied extensively for both price-taker \cite{Salas2018Benchmarking,Loehndorf2010Optimal} and price-maker storage \cite{Wang2017LookAhead, Mohsenianrad2016Coordinated, Ding2017Optimal, Badoual2021Learning}.
These challenges are transferred to the market operator in the case of non-merchant storage. 
Energy markets are cleared sequentially for subsequent finite time horizons, but due to intertemporal constraints, the different market horizons do influence one another. This raises the question of how to account for the subsequent market clearing in the current clearing, and vice versa.
For an overview of market design approaches for markets with intertemporal constraints we refer to \cite{Zhao2020}.


As the market-clearing problem is solved with a finite horizon, the relation with the following market horizons is embedded in the choice of a final state of energy for the storage system. Existing works on market clearing with non-merchant storage often neglect or simplify relations between subsequent market horizons, leading to `myopic' decision-making regarding the state of energy of the storage at the end of horizon. This is pointed out as a research gap by the authors of \cite{Sioshansi2021Energy}.  
A non-myopic, optimal final state of energy is challenging to determine, because it depends on both current and future market horizons. 
The non-merchant storage literature largely bypasses this problem using the following common simplifying assumptions.
Most works do not impose any constraints on the final state of energy in the storage \cite{Taylor2014Financial,Munoz2017Financial,jiang2023duality,weibelzahl2018effects}. 
Another common simplification is to enforce that the storage state of energy is equal at the start and end of a horizon \cite{Crespo2016Impact, Zhang2020Energy, Vespermann2021Access}. 
A possible state-of-energy profile resulting from either of these myopic approaches is plotted in blue in Figure \ref{fig:intrp}, whereas an example optimal profile given perfect information about future market intervals is depicted in black. 
Such simplifying assumptions can lead to suboptimal use of the storage, loss of social welfare, and market inefficiencies, which we will show in this paper. These effects are sometimes overlooked in existing works that consider market properties within a single horizon \cite{jiang2023duality}.
As an alternative, a `future-aware' method, in red in Figure~\ref{fig:intrp}, could approach the perfect profile if informed end-of-horizon decisions were made.

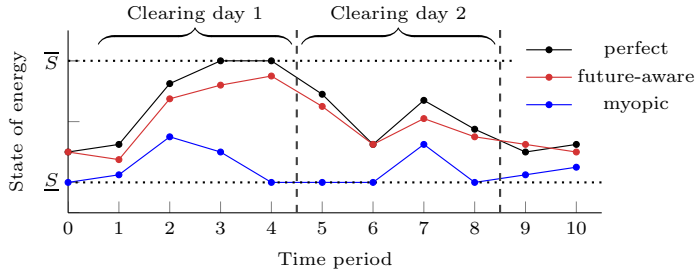
\begin{figure}
\centering
    \begin{tikzpicture}[thick,scale=1,font=\scriptsize]
    \pgfplotsset{ymin=0, ymax=6, xmin=0, xmax=10.5}
    \begin{axis}[xtick = {0,1,2,3,4,5,6,7,8,9,10},
    extra y ticks={1,5},
    extra y tick labels={$\underline{S}$, $\overline{S}$},
    ytick = {0,3,5}, 
    yticklabels={,,},
    xlabel near ticks, 
    ylabel near ticks, 
    ylabel={State of energy }, 
    xlabel={Time period}, 
    axis y line*=left,
    axis x line*=bottom,
    label style={font=\footnotesize},
    tick label style={font=\footnotesize}, 
    legend style={draw=white, name = ConvD_leg_pos, font=\footnotesize, fill=white, anchor = north east, at={(1.2,1)} }, 
    width=8.6cm, height=4cm]
        \addplot [black, line width=0.15mm, mark=*, mark size=1.1, name path=A] 
        table [x index = 0, y index = 3] {introfig.dat};
        \addlegendentry{perfect}
        \addplot [color_ConvD!120, line width=0.15mm, mark=*, mark size=1.1, name path=A] 
        table [x index = 0, y index = 1] {introfig.dat};
        \addlegendentry{future-aware}
        \addplot [blue, line width=0.15mm, mark=*, mark size=1.1, name path=A] 
        table [x index = 0, y index = 2] {introfig.dat};
        \addlegendentry{myopic}
        \addplot[thick, samples=50, name path=Soverbar, black, dotted] coordinates {(0,5)(11.0,5)};
        \addplot[thick, samples=50, name path=Soverbar, black, dotted] coordinates {(0,1)(11.0,1)};
        \addplot[thick, samples=50, name path=daysplit, black!75, dashed] coordinates {(4.5,0)(4.5,10)};
        \addplot[thick, samples=50, name path=daysplit2, black!75, dashed] coordinates {(8.5,0)(8.5,10)};
        \path (axis cs:0, 0) coordinate (origin);
        \path (axis cs:3.77, 1.1) coordinate (0);
        \path (axis cs:6.23, 1.1) coordinate (2);
        \path (axis cs:0.6, 5.55) coordinate (D1_1);
        \path (axis cs:4.4, 5.55) coordinate (D1_2);
        \path (axis cs:4.6, 5.55) coordinate (D2_1);
        \path (axis cs:8.4, 5.55) coordinate (D2_2);
        \path (axis cs:8.6, 5.55) coordinate (D3_1);
        \path (axis cs:10.4, 5.55) coordinate (D3_2);
        
    \end{axis}
    \begin{scope}[decoration={calligraphic brace, amplitude=6pt}]
        \draw[thick,decorate] (D1_1) -- (D1_2) node[midway,above=1ex]{\footnotesize Clearing day 1};
        \draw[thick,decorate] (D2_1) -- (D2_2) node[midway,above=1ex]{\footnotesize Clearing day 2};
    \end{scope}
    \end{tikzpicture}
\caption{Illustrative example: the effect of (simplifying) assumptions on the state of energy of the storage over several market clearing horizons. }
\label{fig:intrp}
\end{figure}

Several previous works do consider setting a future-aware end-of-horizon level for the non-merchant storage, but this has only been done using rolling-horizon methods \cite{Zhao2020, Hua2019Pricing, Chen2021Pricing}.
However, problems can even occur when the final state of energy is set perfectly, both in rolling-horizon or other kinds of look-ahead markets. They result from the fact that subsequent optimization horizons are not aware of losses and gains incurred by market participants in past horizons \cite{Hua2019Pricing}. 
As a result, it can happen that the storage improves social welfare, while not recovering its cost.
Solutions to this problem have been proposed in a rolling-horizon setting, and with a focus on ramping constraints \cite{Hua2019Pricing, Guo2021Pricing}. 

The end-of-horizon issue for non-merchant storage is systematically overlooked, as illustrated by many studies that make simplifying assumptions on the final state of energy. 
To draw attention to this, we aim to convey two key messages in a rigorous manner. 
Our first message is that it is important to set a future-aware end-of-horizon storage level, instead of making simplifying assumptions. 
Second, we show that a new problem arises in the absence of the simplifying assumptions, namely, that market prices in subsequent market horizons may fail to reflect the value of storage. This problem is essential to address, as the market can fail to provide dispatch-following incentives and cost recovery for the storage, as we will show in this work. 
Previous works have touched upon these issues, but only in a rolling-horizon setting, and with a focus on ramping constraints. 
Instead of using a rolling-horizon approach, we consider a more general formulation where the final storage level is constrained to take a certain value, to be determined for each market clearing. 

Our first contribution is to provide intuitive understanding of our two key messages using two illustrative examples. 
Our second and main contribution is the analysis of market properties under simplifying assumptions on the storage system final state of energy, as well as in the absence of these assumptions.
To support the first message, we prove that such common simplifying assumptions ensure cost recovery for the storage, but lead to market inefficiencies. 
Related to the second message, we provide sufficient conditions in Theorem~\ref{thrm:main} under which the market prices in subsequent market horizons will fail to reflect the value of storage, and thus fail to provide dispatch-following incentives and possibly cost recovery for the storage. 
A final contribution of this work is a proposed market-clearing procedure with non-merchant storage which is efficient and ensures cost recovery, given perfect foresight about future market-clearing parameters.

The remainder of this paper is organized as follows. In Section~\ref{sec:2}, we introduce the set-up of this work, including the market-clearing model and pricing scheme. Section~\ref{sec:3} analyzes the impact of common assumptions on market properties. In Section~\ref{sec:4}, we show why and under which conditions dispatch-following incentives and cost recovery for the storage are not ensured, even when the final state of energy is set to its optimal value. In Section~\ref{sec:5}, we propose a method for ensuring that the desirable market properties hold, in the perfect foresight setting. Finally, Section~\ref{sec:conclusion} concludes the paper, followed by several appendices.

\section{Market Model and Definitions} \label{sec:2}

We formulate the market-clearing problem with a non-merchant storage and detail our assumptions. Hereafter, \textit{market horizon} refers to a set of time periods covered by a given market clearing. For example, a day-ahead market would have a market horizon of 24 hours. The next market horizon would include the following 24 hours. We disregard what happens in other markets trading for the same day, e.g., intraday and balancing. We first introduce the storage model in Section \ref{sec:stormod} before proceeding to the market formulation in Section \ref{sec:marketmod}. Next, we formulate the dual problem and discuss the pricing mechanism in Section~\ref{sec:dual}. Finally, we define market properties in Section~\ref{sec:markprop}. Nomenclature and the rules followed for notations are available in Appendix~\ref{ap:nomencl}.

\subsection{Storage Model}\label{sec:stormod}
We assume there is a single storage system. This is equivalent to including multiple storage systems, as we do not consider grid constraints. 
The storage system is described by
\begin{subequations} \label{prob:storage}
\allowdisplaybreaks
\begin{align}
    \label{eq:stor_bound} & 0 \leq e_t \leq \overline{E}, \quad  \forall t \in \mathcal{T}  && :(\underline{\nu}_t, \overline{\nu}_t) \\
    \label{eq:stor_bal_t} & e_t = e_{t-1} + b_t , \quad \forall t \in \mathcal{T}\setminus \{1\}  && :(\rho_t)\\
    \label{eq:stor_bal_1} & e_1 = E^{\mathrm{init}} + b_1 && :(\rho_1) \\
    \label{eq:stor_end} & e_\text{H} = E^{\mathrm{end}} && :(\xi)\,.
\end{align}
\end{subequations}

Here and in the following, dual variables are indicated within parentheses on the right of the corresponding constraint.
For time step $t$ in the market-clearing horizon $\mathcal{T}$, the state of energy of the storage system is a decision variable denoted by $e_t$. It is bounded between 0 and the storage capacity $\overline{E}$ in \eqref{eq:stor_bound}.
The variable $b_t$ represents the energy charged ($b_t>0$) or discharged ($b_t<0$). The storage energy balance is described by \eqref{eq:stor_bal_t} and \eqref{eq:stor_bal_1}, where $E^{\mathrm{init}}$ is the initial state of energy. We omit charging and discharging losses, so that a single variable for charging and discharging suffices. The inclusion of charging and discharging losses would change the dual problem slightly, but as will be argued later on, it would not alter our main message. We disregard charging and discharging limits, since they make the formulations and derivations heavier by introducing additional dual variables, while they do not affect conclusions. Our storage-system model is stylized as we focus on the time-linking aspect in \eqref{eq:stor_bal_t}. The representation of the storage system would be more detailed in a practical market-clearing model. For a discussion on the storage model assumptions, we refer to \cite{Pozo2014Unit} and \cite{garifi2018control}.
The final state of energy for $t=H = |\mathcal{T}|$ can be set to a predefined value $E^{\mathrm{end}}$ with \eqref{eq:stor_end}.
Constraints \eqref{eq:stor_bal_1} and \eqref{eq:stor_end} generalize what is found in the literature, where the initial level is often assumed to be equal to zero and the final level is unconstrained, or where $E^{\mathrm{end}} = E^{\mathrm{init}}$.

The best value of $E^{\mathrm{end}}$ could be determined in many different ways, e.g., rolling horizon, online learning, reinforcement learning, perfect foresight, etc. It is a topic in itself and we will not address it here. In Section~\ref{sec:4}, we consider the situation where $E^{\mathrm{end}}$ is chosen optimally, i.e., with perfect foresight of future market horizons. 
Our formulation is general in the way that it includes various approaches proposed in the literature as special cases. Alternatives to imposing the final level with \eqref{eq:stor_end} could be to have a rolling-horizon set-up or to include an estimated value of stored energy in the objective function of the market clearing \cite{Ding2017Optimal}. Similar end-of-horizon problems also arise in those settings. 

\subsection{Market-Clearing Formulation}\label{sec:marketmod}
In the market-clearing problem with non-merchant storage, the storage does not submit price-quantity bids. However, storage operational constraints~\eqref{eq:stor_bound}-\eqref{eq:stor_bal_1}, and in some cases \eqref{eq:stor_end}, are included in the clearing. The storage is used as an asset to move cheap energy between time periods, similar to the way that power lines can be included in the market clearing to move energy between nodes or zones. As we focus on the time-linking effects of storage, network constraints are excluded for simplicity. The addition of network constraints would, however, not alter our main results. Storage can also modify the dispatch between nodes of the system and have an impact on line congestions, which is studied in \cite{weibelzahl2018effects}. 

We consider two versions of the market clearing with non-merchant storage, namely, a free and a constrained market. In the \textit{constrained} market, \eqref{eq:stor_end} is included in the optimization problem, which gives the market operator the option to enforce an optimal end-of-horizon storage level by considering future market horizons. In the \textit{free} market on the other hand, this constraint is omitted. 
The constrained market-clearing problem for the time periods $t \in \mathcal{T}$, denoted by $\textbf{C}(\mathcal{T})$, is
\begin{subequations} \label{prob:simplified_dispatch}
\allowdisplaybreaks
\begin{align}
    \label{eq:mc_obj}\max_{\mathbf{x}} \quad & \sum_{t \in \mathcal{T}}  \left ( \sum_{l \in \mathcal{L}} U_{lt} d_{lt} - \sum_{g \in \mathcal{G}} C_{gt} p_{gt}  \right )\\
    \label{eq:energy_bal}
    \text{s.t.} \quad & \sum_{l \in \mathcal{L}} d_{lt} + b_t - \sum_{g \in \mathcal{G}} p_{gt} = 0 , \quad  \forall t \in \mathcal{T}   \quad :(\lambda_t)\\
    \label{eq:gen_bound}& 0 \leq p_{gt} \leq \overline{P}_{gt}, \quad  \forall g \in \mathcal{G}, t \in \mathcal{T}   \quad :(\underline{\mu}_{gt}, \overline{\mu}_{gt})\\
    \label{eq:load_bound}& 0 \leq d_{lt} \leq \overline{D}_{lt}, \quad  \forall l \in \mathcal{L}, t \in \mathcal{T}   \quad :(\underline{\chi}_{lt}, \overline{\chi}_{lt})\\
    &\eqref{eq:stor_bound}-\eqref{eq:stor_end} \, .
\end{align}
\end{subequations}
Here, $\mathbf{x}$ is the vector grouping all primal variables, which are indicated with Roman lowercase letters. The generators are indexed by $g \in \mathcal{G}$, and loads by $l \in \mathcal{L}$. The production of $g$ at $t$ is a decision variable $p_{gt}$, and similarly $d_{lt}$ gives the demand of $l$ at $t$. The individual bid and offer prices are $U_{lt}$ and $C_{gt}$. These correspond to the demand utility and the generation cost under the assumption that participants bid truthfully.
The objective function~\eqref{eq:mc_obj} is to maximize the social welfare, calculated as the difference between total demand utility and total generation cost over the given time horizon.
Constraint \eqref{eq:energy_bal} enforces balance between production and demand at each time $t$, including the charged or discharged energy.
Production and demand limits are enforced in \eqref{eq:gen_bound} and \eqref{eq:load_bound}, with maximum levels $\overline{P}_{gt}$ and $\overline{D}_{lt}$, respectively.

The free market-clearing problem for the time periods $t \in \mathcal{T}$, denoted by $\textbf{F}(\mathcal{T})$, is obtained from the constrained problem by removing \eqref{eq:stor_end}.

\subsection{Dual Problem and Pricing}\label{sec:dual}

The dual problem $\textbf{CD}(\mathcal{T})$ of \eqref{prob:simplified_dispatch} is 
\begin{subequations} \label{prob:simplified_dispatch_dual}
\allowdisplaybreaks
\begin{align}
    \min_{\boldsymbol{\zeta}} \quad & \sum_{t \in \mathcal{T}}  \left ( \sum_{g \in \mathcal{G}} \overline{P}_{gt}\overline{\mu}_{gt} + \sum_{l \in \mathcal{L}} \overline{D}_{lt} \overline{\chi}_{lt} + \overline{E} \overline{\nu}_t  \right ) + E^{\mathrm{init}} \rho_1 - E^{\mathrm{end}} \xi \\
    \label{eq:dual_gens}
    \text{s.t.} \quad & C_{gt} - \underline{\mu}_{gt} + \overline{\mu}_{gt} - \lambda_t = 0, \quad  \forall g \in \mathcal{G}, t \in \mathcal{T}\\
    \label{eq:dual_loads}
    & - U_{lt} - \underline{\chi}_{lt} + \overline{\chi}_{lt} + \lambda_t = 0, \quad  \forall l \in \mathcal{L}, t \in \mathcal{T}\\
    \label{eq:rholambda}
    & - \rho_t + \lambda_t = 0, \quad  \forall t \in \mathcal{T}\\
    \label{eq:dual_timelink}
    & - \underline{\nu}_t + \overline{\nu}_t + \rho_t - \rho_{t+1} = 0, \quad  \forall t \in \mathcal{T}\setminus \{H\}\\
    \label{eq:dual_timelink_H}
    & - \underline{\nu}_H + \overline{\nu}_H + \rho_H - \xi = 0 \\
    & \underline{\mu}_{gt}, \overline{\mu}_{gt}, \underline{\chi}_{lt}, \overline{\chi}_{lt}, \underline{\nu}_t, \overline{\nu}_t \geq 0 \, .
\end{align}
\end{subequations}
The vector $\boldsymbol{\zeta}$ groups all dual variables.
We include the dual variables associated with the non-negativity constraints of the primal variables as they are used in our subsequent proofs.
We choose the signs of the free variables $\rho_1$ and $\xi$ in a way that will facilitate our derivations. 
The dual problem $\textbf{FD}(\mathcal{T})$ of the free market clearing is obtained from \eqref{prob:simplified_dispatch_dual} by removing $\xi$ and the terms related to it.

A thorough analysis of the dual problem is available in~\cite{jiang2023duality}. One useful interpretation of $\rho_t$ is that it represents the marginal value of having an additional unit of energy stored at the end of hour $t$. Constraint \eqref{eq:rholambda} implies $\rho_t = \lambda_t$ for all $t$. Therefore, we use $\rho$ and $\lambda$ interchangeably in the remainder of this article. The inclusion of charging and discharging losses would change relation \eqref{eq:rholambda} by a factor. 

We consider a pricing system where all market participants, including the storage system, buy and sell energy at the hourly market price $\lambda_t$.
Alternatively, the participants could be paid rents evaluated from other dual variables, as is done in \cite{Taylor2014Financial,Munoz2017Financial, jiang2023duality}. 
The two payment systems are equivalent, as proven in \cite{jiang2023duality} and \cite{Taylor2014Financial} for the case that $E^{\mathrm{init}} = E^{\mathrm{end}} = 0$. 

Constraints \eqref{eq:dual_timelink} and \eqref{eq:dual_timelink_H} establish a relation between the value of $\rho_t$ in subsequent time periods. These constraints show that the market prices $\lambda_t$ in two subsequent hours can only differ if the storage state of energy is at a bound. This can be seen as intertemporal congestion, similar to the way line congestion in network-aware markets may lead to differences in nodal or zonal prices. 

We define the concept of a \textit{time zone} similar to a spatial zone in network-aware markets. A time zone is the longest possible set of consecutive time steps with the same market price. We illustrate the concept of a time zone in Figure \ref{fig:timezone}.

\begin{figure}[!ht]
\center
    \begin{tikzpicture}[thick,scale=1,font=\footnotesize]
    \pgfplotsset{ymin=0, ymax=5.5, xmin=0, xmax=8.5}
    \begin{axis}[
    xtick = {1,2,3,4,5,6,7,8},
    extra y ticks={1,5},
    extra y tick labels={$\underline{S}$, $\overline{S}$},
    ytick = {0,3,5}, 
    yticklabels={,,},
    ylabel near ticks, 
    ylabel={$e_t$ [$\si{{\watt\hour}}$]}, 
    axis y line*=left,
    axis x line*=bottom,
    label style={font=\footnotesize},
    tick label style={font=\footnotesize}, 
    legend pos = south west, 
    legend style={draw=white, name = ConvD_leg_pos, font=\footnotesize, 
        legend cell align={left},
        fill=white, 
        at={(0.5, 0.9)},
		anchor=north,legend columns=-1},
    width=8.6cm, height=3.5cm]
        \addplot [color_ConvD!120, line width=0.15mm, mark=*, mark size=1.1, name path=A] 
        table [x index = 1, y index = 2] {timezone.dat};
        \addplot[thick, samples=50, name path=Soverbar, black, dotted] coordinates {(0,5)(9.0,5)};
        \addplot[thick, samples=50, name path=Soverbar, black, dotted] coordinates {(0,1)(9.0,1)};
    \end{axis}
    \end{tikzpicture}
\begin{tikzpicture}[thick,scale=1,font=\footnotesize]
    \pgfplotsset{ymin=0, ymax=5.5, xmin=0, xmax=8.5}
    \begin{axis}[
    xtick = {1,2,3,4,5,6,7,8},
    ytick = {0, 5}, 
    xlabel near ticks, 
    ylabel near ticks, 
    ylabel={Price $\rho_t$ [\euro/Wh]}, 
    xlabel={Time periods $t\in\mathcal{T}$}, 
    axis y line*=left,
    axis x line*=bottom,
    label style={font=\footnotesize},
    tick label style={font=\footnotesize}, 
    legend pos = south west, 
    legend style={draw=none, name = ConvD_leg_pos, font=\footnotesize, 
    legend cell align={left},fill=none}, 
    width=8.6cm, height=3.5cm]
        \addplot [black, line width=0.15mm, mark=*, mark size=1.1, name path=A] 
        table [x index = 1, y index = 0] {timezone.dat};
        \addplot[thick, samples=50, name path=TZ0a, gray!25, opacity=0.5] coordinates {(.75,0)(0.75,10)};
        \addplot[thick, samples=50, name path=TZ0b, gray!25, opacity=0.5] coordinates {(1.25,0)(1.25,10)};
        \addplot[pattern = north east lines, pattern color=red!50]fill between[of=TZ0a and TZ0b];
        \addplot[thick, samples=50, name path=TZ1a, gray!25, opacity=0.5] coordinates {(1.75,0.05)(1.75,10)};
        \addplot[thick, samples=50, name path=TZ1b, gray!25, opacity=0.5] coordinates {(3.25,0.05)(3.25,10)};
        \addplot[pattern = vertical lines, pattern color=green]fill between[of=TZ1a and TZ1b];
        \addplot[thick, samples=50, name path=TZ2a, gray!25, opacity=0.5] coordinates {(3.75,0.05)(3.75,10)};
        \addplot[thick, samples=50, name path=TZ2b, gray!25, opacity=0.5] coordinates {(6.25,0.05)(6.25,10)};
        \addplot[color=gray!30] fill between[of=TZ2a and TZ2b];
        \addplot[thick, samples=50, name path=TZ3a, gray!25, opacity=0.5] coordinates {(6.75,0)(6.75,10)};
        \addplot[thick, samples=50, name path=TZ3b, gray!25, opacity=0.5] coordinates {(9.25,0)(9.25,10)};
        \addplot[pattern = north west lines, pattern color=blue!50]fill between[of=TZ3a and TZ3b];
        \addplot[thick, samples=50, name path=daysplit, black!75, dotted] coordinates {(4.5,0)(4.5,10)};
        \node at (axis cs:4.1,0.75) {$\mathcal{X}^{\mathrm{d1}}$};
        \node at (axis cs:5.35,0.75) {$\mathcal{X}^{\mathrm{d2}}$};
        \path (axis cs:0, 0) coordinate (origin);
        \path (axis cs:3.77, 1.1) coordinate (0);
        \path (axis cs:6.23, 1.1) coordinate (2);
        \path (axis cs:0.6, 5.55) coordinate (D1_1);
        \path (axis cs:4.4, 5.55) coordinate (D1_2);
        \path (axis cs:4.6, 5.55) coordinate (D2_1);
        \path (axis cs:8.4, 5.55) coordinate (D2_2);
    \end{axis}

    \begin{scope}[decoration={calligraphic brace, amplitude=6pt}]
        \draw[thick,decorate] (0) -- (2) node[midway,above=1ex]{$\mathcal{X}$};
        \draw[thick,decorate] (D1_1) -- (D1_2) node[midway,above=1ex]{$\mathcal{T}^{\mathrm{d1}}$};
        \draw[thick,decorate] (D2_1) -- (D2_2) node[midway,above=1ex]{$\mathcal{T}^{\mathrm{d2}}$};
    \end{scope}
    \end{tikzpicture}
\caption{Bottom: Example of market-clearing price dividing the market horizon into four time zones. Notation for sets of time indices used in later sections is introduced. Top: Storage state-of-energy ($e_t$) profile related to the below market price signal. \label{fig:timezone}}
\end{figure}
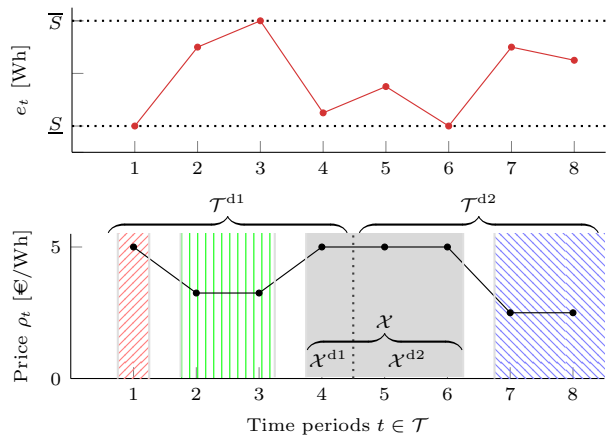

\begin{definition}[Time zone]
A set of time steps $\mathcal{Z}$ with $\min_{t\in\mathcal{Z}} t = z_0$ and $\max_{t\in\mathcal{Z}} t = Z$ is a time zone if and only if 
\begin{enumerate}
    \item $\mathcal{Z}$ only includes consecutive time steps
    \item for all $t\in\mathcal{Z}$ it holds that $\lambda_t = c$ for some constant $c$ 
    \item $\rho_{z_0-1} \neq c$ and $\rho_{Z+1} \neq c$.
\end{enumerate}
Denote the time zone of time step $t$ by $\mathcal{Z}_t$.
\end{definition}

The bottom plot in Figure \ref{fig:timezone} shows four time zones. Within each time zone, the market price $\lambda_t = \rho_t$ is constant. The upper plot in the figure shows an example state-of-energy profile for the storage that can accompany the price profile below. As $\rho$ decreases from $t=1$ to $t=2$, the storage must be at a lower bound at $t=1$, due to constraint \eqref{eq:dual_timelink}. The same constraint implies that when the price increases, e.g., from $t=3$ to $t=4$, the storage must be at an upper bound at $t=3$.

\subsection{Market Properties}\label{sec:markprop}
Here, we define the desirable properties of the market clearing that will be evaluated in the next sections.

Market prices are \textit{dispatch supporting} when no market player desires to unilaterally deviate from the outcomes of the market clearing.
Alternatively, one could say that the market provides \textit{dispatch-following incentives}.
The market is \textit{efficient} if social welfare is maximized and the market provides dispatch-following incentives. 
There is \textit{cost recovery} if the profit of every participant is guaranteed to be non-negative. 
Related works might use other terms to describe similar market properties, e.g., individual rationality.
Another critical market property is \textit{revenue adequacy}, which holds when the market operator has no financial deficit. We do not mention it in the rest of the paper because it is ensured regardless of the assumptions on the final storage level.

In the remainder, we assume that there is perfect competition, in the sense that demands and generators participating in the market bid their true utility and cost. Indeed, we focus on the effects of the storage system on the market properties and disregard the possibility of strategic bidding for the other participants, which would also exist without the storage system.


\section{Market Properties under Common Assumptions on the Final Storage Level} \label{sec:3}

In this section, we analyze the influence that common simplifying assumptions on the final state of energy have on market properties, in particular, cost recovery for the storage system and market efficiency. We evaluate two assumptions.
The first is to not have a constraint on the final level. The market clearing will then schedule the storage in order to maximize the social welfare in the current horizon only, while neglecting future social welfare. As a consequence of not having such a constraint, storage state of energy is guaranteed to be at its lower bound by the end of each market horizon, unless negative prices occur\footnote{If zero market prices occur, there may exist multiple optimal solutions. In principle, it is possible that some of these solutions have a non-empty storage at the end of the market horizon.}.
The second assumption is to assume that storage state of energy is equal at the start and end of a horizon. Often this level is set to zero, which is not necessarily a good choice, as the storage operator could prefer to start the day with some energy available in case the prices are high.
We explore how these assumptions ensure cost recovery over each individual market horizon, but may lead to market inefficiency.

\subsection{Illustrative Example I: Impact of Common Assumptions}\label{sec:sec3a}

In this purely illustrative example, we clear the market for two sequential days of two hours each, identified by the sets $\mathcal{T}^{\mathrm{d1}}$ and $\mathcal{T}^{\mathrm{d2}}$. Two hours is the shortest possible market horizon for which the impact of common assumptions can be illustrated. Similar results can be obtained for longer market horizons. 
The storage capacity is $\overline{E} = 2.5\si{{\watt\hour}}$. 
We consider a single load and two generators, a cheaper one and a more expensive one. The time-dependent inputs are summarized in Table \ref{tab:case-data}. 
The corresponding code is available online \cite{GitHub}.

\begin{table}[ht]
    \caption{Time-dependent inputs for Illustrative Example I of Section \ref{sec:3}}
    \begin{tabular}{cc|cccccc}
     & $t$ & $\overline{D}_{1t} (\si{{\watt\hour}})$ & $\overline{P}_{1t} (\si{{\watt\hour}})$ & $\overline{P}_{2t} (\si{{\watt\hour}})$ & $U_{1t}$ (\si{\euro\per\watt\hour}) & $C_{1t}$ (\si{\euro\per\watt\hour}) & $C_{2t}$ (\si{\euro\per\watt\hour}) \\ \hline
    \multirow{2}{*}{$\mathcal{T}^{\mathrm{d1}}$} & 1 & 0 & 2 & 2 & 12 & 4 & 8  \\
     & 2 & 1 & 2 & 2 & 12 & 5 & 10 \\
    \multirow{2}{*}{$\mathcal{T}^{\mathrm{d2}}$} & 3 & 3 & 2 & 2 & 12 & 2 & 9  \\
     & 4 & 3 & 2 & 2 & 12 & 6 & 11
    \end{tabular}
    \label{tab:case-data}
\end{table}

We look at the profit of the storage system with different assumptions on the initial state of energy $E^{\mathrm{init}}$ and final state of energy $E^{\mathrm{end}}$ on the first day. These results are shown in Table \ref{tab:profit}. In the first case, the initial level is set to 0 and the final level is left free, or equivalently set to zero (if no negative prices occur), which is the most common assumption in the literature. In the second case, the initial and the final levels on each day are equal, which is also a common assumption. The value is set to $1.25\si{{\watt\hour}}$ (half of the capacity). In the third case, the preference of the storage, given perfect information about the second day, is taken into account. The storage starts empty, and the final level on the first day is set to $2.5 \si{\watt\hour}$. The final level on the second day is left free, under the assumption that there is no subsequent market clearing, which is sufficient to show the potential issues. Indeed, in this last case, the storage earns a negative profit on the first day, but this is compensated for by a higher profit on the second day. Note that the storage never obtains a negative daily profit in the other two cases. 

\begin{table}[ht]
    \caption{Profit (\euro) of the storage system for various initial and final state of energy ($\si{{\watt\hour}}$) for {Illustrative Example I}}
    \begin{tabular}{cc|ccc}
                            &
  $t$ &
  \begin{tabular}[c]{@{}c@{}}$E^\text{init}=E^\text{end}$\\ $=0$\end{tabular} &
  \begin{tabular}[c]{@{}c@{}}$E^\text{init}=E^\text{end}$\\ $=1.25$\end{tabular} &
  \begin{tabular}[c]{@{}c@{}}$E^\text{init}=0$ and \\ $E^\text{end}=2.5$\end{tabular} \\ \hline
    \multirow{2}{*}{$\mathcal{T}^{\mathrm{d1}}$}         & 1   & -4                     & -4                      & -10                          \\
             & 2   & 4                      & 4                       & -2.5                         \\
    \multirow{2}{*}{$\mathcal{T}^{\mathrm{d2}}$}         & 3   & -11                    & -11                     & 6                            \\
             & 4   & 11                     & 11                      & 9                            \\ \hline
    Total   $\mathcal{T}^{\mathrm{d1}}$ &     & 0                      & 0                       & -12.5                        \\
    Total   $\mathcal{T}^{\mathrm{d2}}$ &     & 0                      & 0                       & 15                           \\
    Total                   &     & 0                      & 0                       & 2.5                         
    \end{tabular}
    \label{tab:profit}
\end{table}

The total social welfare obtained is equal to $46\euro$ in the first and second cases and $55.5\euro$ in the last case. So not only the profit of the storage is higher in the last case, but the social welfare is too. This difference in social welfare is due to the fact that in the last case, the storage system is used in the first hour of the second day instead of the most expensive generator. In order to charge the storage, the less expensive generator produces more on the first day.

In the following, we generalize these results and show that if the final state of energy is forced to be equal to the initial state of energy or not fixed at all, the daily profit of the storage is never negative. However, these assumptions on the final storage level will most often give a solution that is not optimal for the storage, nor the social welfare.

\subsection{Cost Recovery for the Storage System}
In both market-clearing formulations $\textbf{F}(\mathcal{T})$ and $\textbf{C}(\mathcal{T})$, cost recovery holds for the generators and loads. For the storage system however, cost recovery is not automatically ensured within a single market interval. For example, if the final level is set to be higher than the initial level and the prices are positive and decreasing during the market interval, the storage system pays for charging at the last hour and does not recover that cost in this market interval. However, we show in this section that under certain conditions, cost recovery for the storage within each market interval is ensured. 

\subsubsection{Cost recovery in $\textbf{C}(\mathcal{T})$ with equal initial and final state of energy}
In order to evaluate cost recovery for $\textbf{C}(\mathcal{T})$, we express the profit of the storage as a function of the initial and final state of energy. We use the profit-maximization problem of the storage and strong duality to find this relationship.
Related to the market-clearing model~\eqref{prob:simplified_dispatch}, the profit maximization for the storage is
\begin{subequations} \label{prob:prof}
\allowdisplaybreaks
\begin{align} 
    \label{eq:prof_obj} \max_\mathbf{b,e} \quad & \sum_{t \in \mathcal{T}}  - \lambda_t b_t  \\
    \text{s.t.} \quad & 0 \leq e_t \leq \overline{E}, \quad  \forall t \in \mathcal{T}  & \quad :(\underline{\nu}_t, \overline{\nu}_t) \\
    & e_t = e_{t-1} + b_t , \quad \forall t \in \mathcal{T}, t \neq 1 & \quad :(\rho_t) \\
    & e_1 = E^{\mathrm{init}} + b_1 & \quad :(\rho_1) \\
    & e_H = E^{\mathrm{end}} & \quad :(\xi) \,.
\end{align}
\end{subequations}
The objective of the storage is to maximize its profit \eqref{eq:prof_obj}, considering that it is paid the market price $\lambda_t$ when discharging ($b_t \leq 0$), and it pays the market price $\lambda_t$ when charging ($b_t \geq 0$). The market price is taken as an input. The dual problem of \eqref{prob:prof} is
\begin{subequations}
\allowdisplaybreaks
\begin{align}
    \min_{\boldsymbol{\rho, \nu}} \quad & \sum_{t \in \mathcal{T}} \overline{E} \overline{\nu}_t + E^{\mathrm{init}} \rho_1 - E^{\mathrm{end}} \xi \\ 
    \text{s.t.} \quad & \lambda_t - \rho_t = 0, \quad  \forall t \in \mathcal{T}\\ 
    \label{eq:prof_d2} & -\underline{\nu}_t + \overline{\nu}_t + \rho_t - \rho_{t+1} = 0, \quad  \forall t \in \mathcal{T}, t \neq H\\
    \label{eq:prof_d3} & -\underline{\nu}_H + \overline{\nu}_H + \rho_H - \xi = 0 \\
    & \underline{\nu}_t, \overline{\nu}_t \geq 0 \,.
\end{align}
\end{subequations}
At optimality, strong duality holds, and the values obtained for the primal and dual objective functions are equal (optimal variable values are denoted by an asterisk). The profit at optimality is thus
\begin{equation}
     \sum_{t \in \mathcal{T}}  - \lambda_t b_t^* =  \sum_{t \in \mathcal{T}} \overline{E} \overline{\nu}^*_t + E^{\mathrm{init}} \rho_1^* - E^{\mathrm{end}} \xi^* \,. 
     \label{eq:profit}
\end{equation}

Adding \eqref{eq:prof_d2} for all $t$ to \eqref{eq:prof_d3}, and multiplying by $E^{\mathrm{end}}$, we get
\begin{equation}\label{eq:e_fix_muliply}
     E^{\mathrm{end}} \xi^* =  - \sum_{t \in \mathcal{T}}  E^{\mathrm{end}} \underline{\nu}^*_t + \sum_{t \in \mathcal{T}}  E^{\mathrm{end}} \overline{\nu}^*_t + \rho_1^* E^{\mathrm{end}} \,.
\end{equation}

Substituting \eqref{eq:e_fix_muliply} in the right side of \eqref{eq:profit}, we obtain
\begin{equation}
    \sum_{t \in \mathcal{T}}  - \lambda_t b_t^* =  \sum_{t \in \mathcal{T}} (\overline{E}-E^{\mathrm{end}}) \overline{\nu}^*_t + \sum_{t \in \mathcal{T}} E^{\mathrm{end}} \underline{\nu}^*_t + \rho_1^* (E^{\mathrm{init}} - E^{\mathrm{end}}) \,.
\end{equation}

This equality shows that if the final level is set equal to the initial level, the profit of the storage system is non-negative, due to the non-negativity of $\overline{\nu}_t$ and $\underline{\nu}_t$, and the fact that $\overline{E}\geq E^{\mathrm{end}}$. However, $E^{\mathrm{init}} = E^{\mathrm{end}}$ is a special case. In general, we do not have guarantees that the profit of the storage system is non-negative, and in fact it is straightforward to find counterexamples, as was provided in Section \ref{sec:sec3a}.

\subsubsection{Cost recovery in $\textbf{F}(\mathcal{T})$}

We show here that not having a constraint on the final state of energy is also a special case for which the profit of the storage system is non-negative. For $\textbf{F}(\mathcal{T})$, constraint \eqref{eq:prof_d3} is modified to 
\begin{equation}
      \label{eq:dual_short} -\underline{\nu}_H + \overline{\nu}_H + \rho_H = 0 \,.
\end{equation}
Strong duality from \eqref{eq:profit} simplifies to
\begin{equation}
     \sum_{t \in \mathcal{T}}  - \lambda_t b_t^* =  \sum_{t \in \mathcal{T}} \overline{E} \overline{\nu}^*_t + E^{\mathrm{init}} \rho_1^* \,.
     \label{eq:profit_short}
\end{equation}

Summing all constraints in \eqref{eq:prof_d2} and \eqref{eq:dual_short} for the optimal solution, we get
\begin{equation}\label{eq:multiply}
      \sum_{t \in \mathcal{T}} \overline{\nu}^*_t =  \sum_{t \in \mathcal{T}} \underline{\nu}^*_t - \rho_1^* \,.
\end{equation}
Since $\underline{\nu}^*_t$ is non-negative,
\begin{equation}\label{eq:ineq_nu}
      \sum_{t \in \mathcal{T}} \overline{\nu}^*_t \geq - \rho_1^* \,.
\end{equation}
Multiplying both sides by $E^{\mathrm{init}}$, and using that $\overline{E} \geq E^{\mathrm{init}}$, we derive
\begin{align}\label{eq:ineq_level}
      \overline{E} \sum_{t \in \mathcal{T}} \overline{\nu}^*_t \geq E^{\mathrm{init}} \sum_{t \in \mathcal{T}} \overline{\nu}^*_t \geq - E^{\mathrm{init}} \rho_1^* \,.
\end{align}
We conclude
\begin{equation}\label{eq:ineq_level2}
      \sum_{t \in \mathcal{T}}  - \lambda_t b_t^* = \overline{E} \sum_{t \in \mathcal{T}} \overline{\nu}^*_t + E^{\mathrm{init}} \rho_1^* \geq 0 \,.
\end{equation}
In other words, the profit of the storage system is certainly non-negative when the final state of energy is unconstrained.

The absence of cost recovery over a single market interval is not necessarily an issue: the storage may still recover its cost in the subsequent market intervals. In that sense, it would be more relevant to evaluate cost recovery for the storage over an infinitely long time horizon (or equivalent). However, if market intervals are considered completely separate, cost recovery cannot be ensured. 
Setting the initial and final levels to the same value is a way to overcome this, but it has an impact on other market properties, as we discuss next.

\subsection{Market Efficiency}
We now show that enforcing the initial level to be equal to the final level can pose problems in terms of market efficiency.
We argue that under this simplification, the storage system could have an incentive to unilaterally deviate from the market-clearing outcomes. Indeed, the storage system profit-maximization problem would actually take into account a longer horizon than just one market interval, to best utilize the potential of temporal arbitrage. Ideally, the planning horizon for the owner of the storage system would also correspond to this longer horizon. The storage profit maximization should be evaluated over a longer horizon $\mathcal{T^{\mathrm{long}}}$, where $H^{\mathrm{long}}>H$, with $H=|\mathcal{T}|$ and $H^{\mathrm{long}}=|\mathcal{T^{\mathrm{long}}}|$:
\begin{subequations} \label{prob:profit_T}
\begin{align} 
    \max_\mathbf{b, e} \quad & \sum_{t \in \mathcal{T^{\mathrm{long}}}}  - \lambda_t b_t  \\
    \text{s.t.} \quad & 0 \leq e_t \leq \overline{E}, \quad  \forall t \in \mathcal{T^{\mathrm{long}}} \\
    & e_t = e_{t-1} + b_t , \quad \forall t \in \mathcal{T^{\mathrm{long}}}, t \neq 1 \\
    & e_1 = E^{\mathrm{init}} + b_1 \,.
\end{align}
\end{subequations}
Thus, if the final level at the end of the market interval is not set to a value found in the storage profit-maximization problem, i.e. if $E^{\mathrm{end}}\neq e_H^*$, where $e_H^*$ is an optimal solution for \eqref{prob:profit_T}, the market output will not be optimal for the storage system and the storage will have an incentive to deviate.

The social welfare is also impacted, as shown in Illustrative Example I from Section~\ref{sec:sec3a}. Clearing a set of market intervals together will always give the best value of the social welfare over the entire set of time periods included, since the social welfare is maximized. When we instead clear these market intervals sequentially without setting the final value of the storage properly for every single clearing, the social welfare may be reduced.

In conclusion, it is recommended to allow system operators to set the final storage level to a sensible level, based on information about future market intervals. 
However, doing so brings new challenges. In the following, we analyze this type of market clearing and discuss the duality issues that arise. 

\section{Market Properties of the Multi-Interval Market Clearing} \label{sec:4}

While the previous section focused on market properties within a single market interval, this section establishes several results for the multi-interval case. 
The main aim is to determine the effect of splitting the market clearing in multiple intervals on the optimal primal and dual solutions. 
Ideally, the splitting should not affect the solutions. 
We show that even though the primal solution may be unaffected, this is not necessarily the case for the dual solution.

First, we specify the formulations used in the remainder of this work in Section \ref{sec:def4}, including the definition of the \textit{full-horizon} and \textit{split-horizon} problems. 
In Section \ref{sec:illex} we use an illustrative example to show how solutions to these optimization problems may differ, and why this deserves attention. 
The following sections formalize and generalize the results of this illustrative example. 
In Section \ref{sec:primalequiv}, we show that the primal solutions to these two problems are equivalent. 
However, as will be shown in Section \ref{sec:wrongduals}, market efficiency is not ensured in the split-horizon problem, while it does hold for the full-horizon problem. We provide conditions under which market inefficiencies will occur in the split-horizon problem.  

\subsection{Definition of the Full-Horizon and Split-Horizon Problems}\label{sec:def4}
We consider a horizon $\mathcal{T}$ that consists of two market intervals, assumed to be days. The time periods of first and second days are collected in the respective sets $\mathcal{T}^{\mathrm{d1}}$ and $\mathcal{T}^{\mathrm{d2}}$, where $H = |\mathcal{T}^{\mathrm{d1}}|$ and $H+1$ is the first hour of $\mathcal{T}^{\mathrm{d2}}$. The end of the entire horizon is $T = |\mathcal{T}|$. 
We define the \textit{full-horizon} problem as the market clearing for the entire horizon $\mathcal{T}$. The full-horizon optimization model is given by $\textbf{F}(\mathcal{T})$, as already defined in~Section \ref{sec:marketmod}.

We use the term \textit{split-horizon} problem for the case where we clear the market sequentially for $\mathcal{T}^{\mathrm{d1}}$ and then for $\mathcal{T}^{\mathrm{d2}}$. The split-horizon optimization model is given by $\textbf{C}(\mathcal{T}^{\mathrm{d1}})$ for the first and $\textbf{F}(\mathcal{T}^{\mathrm{d2}})$ for the second day. In the split-horizon market with \textit{perfect foresight}, the final storage level on $\mathcal{T}^{\mathrm{d1}}$ is set optimally. That is, $E^{\mathrm{end}}$ is equal to the final optimal storage level $e_H^*$ when solving the full-horizon problem. 
For fair comparison, the storage level at the end of the second day is unconstrained in both the full-horizon and the split-horizon problem.
The two optimization problems for the split-horizon problem are independent, besides the fact that one must choose the parameters right, which means that $E^{\mathrm{end}}$ on the first day must equal the initial storage level $E^{\mathrm{init}}$ on the second day. As the optimization problems are independent, they can be solved in an equivalent combined optimization without changing their solution.
The term \textit{split-horizon problem} refers to this combined optimization.
The objective functions of the full-horizon and split-horizon problems are equal.

We assume that we can disregard what happens after time $t = T$, i.e. we assume that there is a finite horizon for which the market with non-merchant storage needs to be cleared. 
This simplification is justified by the need for a perfect baseline to which we can compare the split-horizon solution. 
Over a finite horizon, the results from Section \ref{sec:3} show that the market clearing with non-merchant storage achieves the optimal solution, and cost recovery is ensured for the storage. Therefore, we can consider this full-horizon market clearing as a perfect benchmark. 

Optimal values in the full- and split- horizon problem are denoted with $*$ and $'$, respectively.
Our proofs use the Karush-Kuhn-Tucker (KKT) conditions of the full- and split-horizon problems, in particular those corresponding to dual constraints from \eqref{prob:simplified_dispatch_dual}. For completeness, we provide the full sets of KKT conditions in Appendix \ref{ap:KKT}.

\subsection{Illustrative Example II: Dual Multiplicity}\label{sec:illex}
Suppose the market is cleared for two sequential hours. 
The market participants include a single load, two generators, and a non-merchant storage with capacity $\overline{E} = 2.5 \si{{\watt\hour}}$ and initial level $E^\text{init} = 0$. 
The time-dependent parameters, demand and production limits, utility, and cost, are summarized in Table \ref{tab:time-inputs}, and will also be visible in market-clearing diagrams.
The corresponding code is available online \cite{GitHub}. 

\begin{table}[ht]
    \caption{Time-dependent inputs for Illustrative Example II of Section \ref{sec:4}}
    \label{tab:time-inputs}
    \begin{tabular}{c|cccccc}
    $t$ & $\overline{D}_{1t} (\si{{\watt\hour}})$ & $\overline{P}_{1t} (\si{{\watt\hour}})$ & $\overline{P}_{2t} (\si{{\watt\hour}})$ & $U_{1t}(\si{\euro\per\watt\hour})$ & $C_{1t}(\si{\euro\per\watt\hour})$ & $C_{2t}(\si{\euro\per\watt\hour})$ \\
    \hline
         1 & 0 & 2 & 2 & 12 & 5 & 10 \\
         2 & 3 & 2 & 2 & 12 & 2 & 9
    \end{tabular}
\end{table}

In order to evaluate the effect of splitting the market clearing on primal and dual solutions, we solve both the split-horizon problem with perfect foresight and the full-horizon problem. We thus show that the lack of dispatch support is not due to imperfect information, but an inherent effect of the market splitting. The optimal values for relevant variables are summarized in Table \ref{tab:cs-results}. The primal variables of the full- and split-horizon problems have equal optimal values and therefore appear in the table only once. 
\begin{table}[b]
\caption{Selected optimal primal and dual values in Illustrative Example II }
\label{tab:cs-results}
    \begin{tabular}{c|cccccc}
    $t$ & $e_t (\si{{\watt\hour}})$ & $ d_{1t} (\si{{\watt\hour}})$ & $ p_{1t} (\si{{\watt\hour}})$ & $ p_{2t} (\si{{\watt\hour}})$ & $ \lambda_{t}^*(\si{\euro\per\watt\hour})$  & $\lambda_{t}'(\si{\euro\per\watt\hour})$ \\
    \hline
         1 & 1 & 0 & 1 & 0 & 5 & 5 \\
         2 & 0 & 3 & 2 & 0 & 5 & [2, 9]
    \end{tabular}
\end{table}

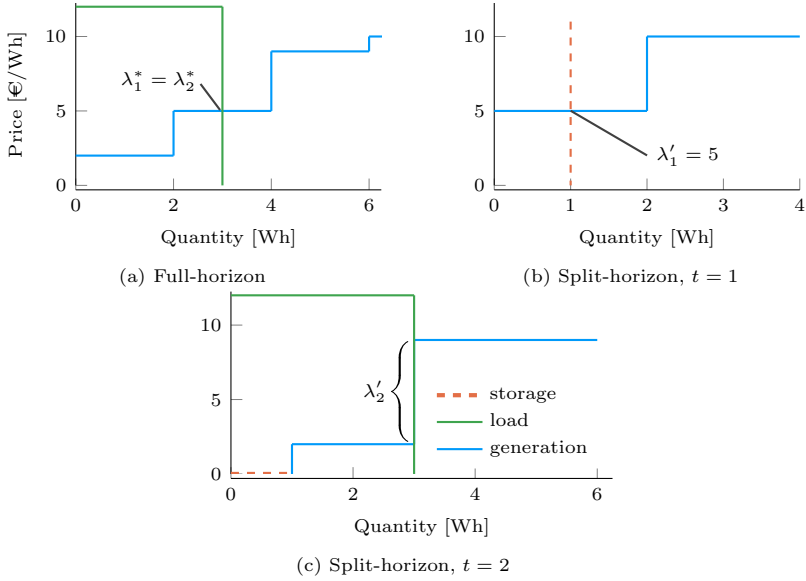
\begin{figure*}[t]
  \centering
    \subfloat[Full-horizon \label{fig:illex_full}]
    {\begin{tikzpicture}[thick,scale=1,font=\footnotesize]
    \pgfplotsset{ymin=0, ymax=12.25, xmin=0, xmax=6.25}
    \begin{axis}[
    xtick = {0,2,4,6,8},
    ytick = {0,5,10}, 
    ylabel near ticks, 
    ylabel={Price $[\euro / \si{{\watt\hour}}]$}, 
    xlabel={Quantity $[\si{{\watt\hour}}]$}, 
    axis y line*=left,
    axis x line shift=0.3 ,
    axis x line*=bottom,
    label style={font=\footnotesize},
    tick label style={font=\footnotesize}, 
    legend style={draw=white, name = ConvD_leg_pos, font=\footnotesize, 
        legend cell align={left},
        fill=white, legend pos = south east},
    width=5.6cm, height=4cm]
        \addplot[thick, samples=50, julia3] coordinates {(0,12)(3,12)};
        \addplot[thick, samples=50, julia1] coordinates {(0,2)(2,2)};
        \addplot[thick, samples=50, julia1] coordinates {(2,2)(2,5)};
        \addplot[thick, samples=50, julia1] coordinates {(2,5)(4,5)};
        \addplot[thick, samples=50, julia1] coordinates {(4,5)(4,9)};
        \addplot[thick, samples=50, julia1] coordinates {(4,9)(6,9)};
        \addplot[thick, samples=50, julia1] coordinates {(6,9)(6,10)};
        \addplot[thick, samples=50, julia1] coordinates {(6,10)(8,10)};
        \addplot[thick, samples=50, julia3] coordinates {(3,12)(3,0)};
        \path (axis cs:1.7, 7) coordinate (D1_1);
        \path (axis cs:2.55, 6.8) coordinate (D1_1b);
        \path (axis cs:2.95, 5.05) coordinate (D1_2);  
    \end{axis}
        \node at (D1_1) {$\lambda_1^* = \lambda_2^*$};
        \draw[color = black!75] (D1_1b) -- (D1_2);
    \end{tikzpicture}
    }
  \label{fig:illex2full}
  \qquad
  \subfloat[Split-horizon, $t=1$ \label{fig:illex_split_t1}]
    {\begin{tikzpicture}[thick,scale=1,font=\footnotesize]
    \pgfplotsset{ymin=0, ymax=12.25, xmin=0, xmax=4.0}
    \begin{axis}[
    xtick = {0,1,2,3,4,6},
    ytick = {0,5,10}, 
    xlabel={Quantity $[\si{{\watt\hour}}]$}, 
    axis y line*=left,
    axis x line shift=0.3,
    axis x line*=bottom,
    label style={font=\footnotesize},
    tick label style={font=\footnotesize}, 
    legend style={draw=white, name = ConvD_leg_pos, font=\footnotesize, 
        legend cell align={left}, anchor = south west,
        fill=white, at={(axis cs:4,2.5)}},
    width=5.6cm, height=4cm]
        \addplot[thick, samples=50, julia2, dashed] coordinates {(1,11)(1,0)};
        \addplot[thick, samples=50, julia1] coordinates {(0,5)(2,5)};
        \addplot[thick, samples=50, julia1] coordinates {(2,5)(2,10)};
        \addplot[thick, samples=50, julia1] coordinates {(2,10)(4,10)};
        \path (axis cs:1, 5) coordinate (D1_1);
        \path (axis cs:2, 2) coordinate (D1_2);    
    \end{axis}
    \begin{scope}[decoration={calligraphic brace, amplitude=6pt}]
        \node[anchor= west] at (D1_2) {$\lambda_1' = 5$};
        \draw[color = black!75] (D1_2) -- (D1_1);
    \end{scope}
    \end{tikzpicture}}
    \qquad
    \subfloat[Split-horizon, $t=2$ \label{fig:illex_split_t2}]
    {\begin{tikzpicture}[thick,scale=1,font=\footnotesize]
    \pgfplotsset{ymin=0, ymax=12.25, xmin=0, xmax=6.25}
    \begin{axis}[
    xtick = {0,2,4,6,8},
    ytick = {0,5,10}, 
    xlabel={Quantity $[\si{{\watt\hour}}]$}, 
    axis y line*=left,
    axis x line shift=0.3,
    axis x line*=bottom,
    label style={font=\footnotesize},
    tick label style={font=\footnotesize}, 
    legend pos = south west, 
    legend style={draw=white, name = ConvD_leg_pos, font=\footnotesize, 
        legend cell align={left},
        fill=white, legend pos = south east},
    width=6.6cm, height=4cm]
        \addplot[ultra thick, samples=50, julia2, dashed] coordinates {(0,0)(1,0)};
        \addlegendentry{storage}
        \addplot[thick, samples=50, julia3] coordinates {(0,12)(3,12)};
        \addlegendentry{load}
        \addplot[thick, samples=50, julia1] coordinates {(1,0)(1,2)};
        \addlegendentry{generation}
        \addplot[thick, samples=50, julia1] coordinates {(1,2)(3,2)};
        \addplot[thick, samples=50, julia1] coordinates {(3,2)(3,9)};
        \addplot[thick, samples=50, julia1] coordinates {(3,9)(6,9)};
        \addplot[thick, samples=50, julia3] coordinates {(3,12)(3,0)};
        \path (axis cs:2.9, 2.2) coordinate (D1_1);
        \path (axis cs:2.9, 8.9) coordinate (D1_2);    
    \end{axis}

    \begin{scope}[decoration={calligraphic brace, amplitude=6pt}]
        \draw[thick,decorate] (D1_1) -- (D1_2) node[midway,left=1ex]{$\lambda_2'$};
    \end{scope}
    \end{tikzpicture}
}
  \caption{Market-clearing diagram for full- and split-horizon problem in Illustrative Example II. The common legend is provided in Subfigure (c).}
\end{figure*}

In the full-horizon problem, the two hours are cleared together. As the storage is not at a bound after the first hour, the two hours comprise a single time zone, so that the market-clearing diagram for the two hours can be depicted as in Figure~\ref{fig:illex_full}. The unique optimal solution is to charge the storage to $e_1^* = 1\si{{\watt\hour}}$, and discharge it fully at $t=2$. 
As the storage is charging 1 unit and discharging 1 unit within the full horizon, its net power injection is zero, and therefore the storage is not depicted in Figure~\ref{fig:illex_full}.
The corresponding optimal values for the market price are $\boldsymbol{\lambda}^* = [5, 5] \si{\euro\per\watt\hour}$.
The storage improves the optimal value of the objective function by $1 \si{\watt\hour} \cdot (9 - 5) \si{\euro\per\watt\hour}= 4 \euro$ compared to the case without storage, as one unit of generation by the more expensive generator with cost of $9\si{\euro\per\watt\hour}$ at $t=2$ is replaced by the cheaper generator with cost of $5\si{\euro\per\watt\hour}$ at $t=1$.

The split-horizon problem with perfect foresight clears the two hours separately, while enforcing in the first market interval that $E^{\mathrm{end}} = e^*_1 = 1\si{{\watt\hour}}$. The primal solution is equal to that of the full-horizon market --- this result we formalize in Section~\ref{sec:primalequiv}. In the first hour, $\lambda_1' = 5\si{\euro\per\watt\hour}$, as illustrated in the market-clearing diagram in Figure~\ref{fig:illex_split_t1}. 
In the split-horizon market for the second hour, none of the loads and generators are marginal. This is depicted in the market-clearing diagram in Figure~\ref{fig:illex_split_t2}. Therefore, there is no longer a unique dual solution for the split-horizon problem, as all values $\lambda_2'$ in the range $[2, 9] \si{\euro\per\watt\hour}$ are optimal. This interval includes the full-horizon optimum for $\lambda_2$, so the full-horizon dual solution is still optimal for the split-horizon dual problem. This is a general result, as we will prove in Section~\ref{sec:primalequiv}. 

The dual multiplicity that can arise in the split-horizon problem deserves attention for several reasons. If $\lambda_2' \in [2, 5) \si{\euro\per\watt\hour}$ is selected, the market does not provide dispatch-following incentives to the storage, which can increase its profit by decreasing $e_1$. In this example, the storage even earns a negative profit, illustrating that cost recovery is not ensured. At the same time, this dispatch of the storage does achieve the optimal social welfare. Furthermore, as shown in this example, there exists a price that supports the dispatch. 

One could argue that just selecting the `correct', dispatch-supporting dual value for $\lambda_2' = 5\si{\euro\per\watt\hour}$ would solve this problem. However, in order to do so the market operator first of all needs to be {aware} of the possible existence of multiple dual solutions. Solvers usually return only one solution, which is not necessarily the property-preserving one. Second, if dual multiplicity occurs in the first market interval rather than the second, then the market operator cannot yet know which dual solution will preserve market properties.
Finally, selecting a dispatch-supporting dual solution becomes more complicated in an imperfect-foresight setting, where the chosen level $E^{\mathrm{end}}$ may be suboptimal in hindsight. 

\subsection{Shared Solutions of the Full- and Split-Horizon Problems}\label{sec:primalequiv}
In this section, we investigate whether the solutions of the full-horizon are also solutions to the split-horizon problem.  
Lemma \ref{lem:equivprimal} shows that primal solutions to the split-horizon problem with perfect foresight and those to the full-horizon problem are identical. 

\begin{lemma}\label{lem:equivprimal}
Let $e_H^*$ be part of an optimal solution $\mathbf{x}^*$ to the full-horizon problem. Then $\mathbf{x}^*$ is an optimal primal solution to the full-horizon problem if and only if $\mathbf{x}^*$ is an optimal primal solution to the split-horizon problem with $E^\mathrm{end} = e_H^*$.
\end{lemma}
\begin{proof}
In Appendix \ref{ap:equivprimal}.
\end{proof}

Next, Lemma \ref{lem:admitdual} shows that any optimal primal and dual solution to the original problem can be converted to an equivalent solution for the split-horizon problem. 
This correspondence is sometimes overlooked in the literature. For example, \cite{Hua2019Pricing} discusses the existence of dual solutions to the split-horizon problem that are not dispatch supporting, but does not mention that the original dual solution remains a solution to the split-horizon problem too. 
\begin{lemma}\label{lem:admitdual}
Any optimal primal and dual solution pair $\{ \mathbf{x}^*, \boldsymbol{\zeta}^*\}$ to the full-horizon problem is also an optimal solution to the split-horizon problem, with additional split-horizon variable $\xi'$ taking the value $\xi' = \rho_{H+1}^*$.
\end{lemma}
\begin{proof}
In Appendix \ref{ap:admitdual}.
\end{proof}

\subsection{Weak Solutions to the Split-Horizon Problem}\label{sec:wrongduals}
While every optimal dual solution to the full-horizon problem is also an optimal dual solution to the split-horizon problem, the reverse generally does not hold. In this section, we discuss the existence of a weak optimal dual solution admitted by the split-horizon problem, but not by the full-horizon problem.
First, we define this type of dual solution, and show why it is a problem. 
Next, we provide sufficient conditions for the existence of weak dual solutions in the split-horizon problem in our main Theorem \ref{thrm:main}. 

\begin{definition}[Weak dual]\label{def:dual}
A dual solution to the split-horizon problem is \textbf{weak} if the resulting price is not dispatch supporting for the non-merchant storage. 
Such a weak dual exists if and only if one of the following situations occurs.
\begin{enumerate}
    \item $e_H \in (\underline{S}, \overline{E})$ and $\rho_H' \neq \rho_{H+1}'$
    \item $e_H = \underline{S}$ and $\rho_H' < \rho_{H+1}'$
    \item $e_H = \overline{E}$ and $\rho_H' > \rho_{H+1}'$ \,.
\end{enumerate} 
\end{definition}
For a weak dual solution, the storage operator has incentive to unilaterally deviate from the schedule determined in the market. This incentive exists if the solution to the storage's profit maximization problem \eqref{prob:prof} given market prices $\lambda_t$ is different from the schedule determined in the market. 
We distinguish two cases:
\begin{enumerate}
    \item If $\rho_H' = \xi' > \rho_{H+1}'$, the storage can improve its profit by decreasing $e_H'$. This is possible in case $e_H \neq \underline{S}$, i.e. both in situations 1) and 3) of Definition \ref{def:dual}.
    \item If $\rho_H' = \xi' < \rho_{H+1}'$, the storage can improve its profit by increasing $e_H'$. This is possible in case $e_H \neq \overline{E}$, i.e. both in situations 1) and 2) of Definition \ref{def:dual}.
\end{enumerate}
In extreme cases, a weak dual solution can cause the loss of cost recovery for the storage operator. For example, this happens if the storage bought all of the energy $e_H$ at the price $\rho_H'$ and is scheduled to sell all of it for a lower price $\rho_{H+1}'$.

\begin{theorem}\label{thrm:main}
Assume all cost and utility bids are unique \footnote{More specifically, utility bids are unique $U_{lt} = U_{l't'} \iff l = l' \wedge t = t'$, cost bids are unique and $C_{gt} = C_{g't'} \iff g = g' \wedge t = t'$, and $U_{lt} \neq C_{gt'} \, \forall t, t', l, g$. }.
If the optimal solution to the full-horizon problem is such that $H+1$ and $T$ are in different time zones, and
\begin{align}\label{condition_thrm}
    \rho_H^* = \rho_{H+1}^* \,,
\end{align}
then the split-horizon problem with perfect foresight admits a weak dual solution.
\end{theorem}
\begin{proof}
We construct a weak dual solution to the split-horizon problem, based on the given dual solution to the full-horizon problem.
More precisely, we derive values of $\delta>0$ for which 
\begin{align}
    \rho_{H+1}' = \rho_{H+1}^* + \delta 
\end{align}
is part of an optimal dual solution $\boldsymbol{\zeta}'$ that satisfies the KKT conditions of the non-myopic split-horizon problem with perfect foresight. We construct $\delta$ in such a way that $\rho^*_H$ can remain unchanged, i.e. $\rho_H' = \rho^*_H$.
In this proof, we refer to KKT conditions coming from the dual and primal problems, which have been presented in previous sections. 

By \eqref{condition_thrm}, $H$ and $H+1$ belong to the same time zone $\mathcal{X}$. 
Denote the time periods in $\mathcal{X}$ that are part of day 1 by the set $\mathcal{X}^{\mathrm{d1}}$, and those that are part of day 2 by $\mathcal{X}^{\mathrm{d2}}$.
Note that $\mathcal{X}^{\mathrm{d1}} \subseteq \mathcal{T}^{\mathrm{d1}}$ and $\mathcal{X}^{\mathrm{d2}} \subseteq \mathcal{T}^{\mathrm{d2}}$.
These sets are illustrated in Figure \ref{fig:timezone}.

Due to uniqueness of cost and utility bids, there can at most be one marginal load or generator in $\mathcal{X}$ for a single time period $t\in\mathcal{X}$, which is part of either $\mathcal{X}^{\mathrm{d1}}$ or $\mathcal{X}^{\mathrm{d2}}$. Therefore, at least one of $\mathcal{X}^{\mathrm{d1}}$ and $\mathcal{X}^{\mathrm{d2}}$ does not have a marginal load or generator. We assume that this is the case for $\mathcal{X}^{\mathrm{d2}}$. The proof is similar for the other case.

By uniqueness of cost and utility bids, $\rho_H^*$ can at most be equal to a single cost or utility bid within $\mathcal{X}$.
If there is no marginal load or generator in $\mathcal{X}$, it can happen that $\rho_H^*$ is not equal to any cost or utility bid.

From the KKT conditions of the full-horizon problem corresponding to dual constraints \eqref{eq:dual_gens} and \eqref{eq:dual_loads}, for scheduled loads and generators in $\mathcal{X}^{\mathrm{d2}}$ it holds that
\begin{align}\label{eq:utilcon_new}
    \overline{\chi}^*_{lt} &=  U_{lt} - \rho_{t}^* \geq 0 &&\text{for } \forall t \in \mathcal{X}^{\mathrm{d2}},\, l \in \mathcal{L}_t^{\mathrm{+}} \\
    \label{eq:costcon_new}
    \overline{\mu}^*_{gt} &= \rho_{t}^* -  C_{gt} \geq 0 &&\text{for } \forall t \in \mathcal{X}^{\mathrm{d2}},\, g \in \mathcal{G}_t^{\mathrm{+}} \,,
\end{align}
where $\mathcal{L}_t^{\mathrm{+}} = \{ l \in \mathcal{L}|\, d^*_{lt} > 0\} $ and $\mathcal{G}_t^{\mathrm{+}} = \{ g \in \mathcal{G}|\, p^*_{gt} > 0\} $.
The inequalities are actually strict for all but a single scheduled load or generator, except possibly the load or generator that has a cost bid equal to $\rho_H^*$. Similarly, for unscheduled loads and generators in $\mathcal{X}^{\mathrm{d2}}$, 
\begin{align}\label{eq:utilcon_new_2}
    \underline{\chi}^*_{lt} &=   \rho_{t}^* - U_{lt} > 0 &&\text{for } \forall t \in \mathcal{X}^{\mathrm{d2}},\, l \in \mathcal{L}_t^{\mathrm{0}} \\
    \label{eq:costcon_new_2}
    \underline{\mu}^*_{gt} &=  C_{gt} - \rho_{t}^* > 0 &&\text{for } \forall t \in \mathcal{X}^{\mathrm{d2}},\,  g \in \mathcal{G}_t^{\mathrm{0}} \,,
\end{align}
where $\mathcal{L}_t^{\mathrm{0}} = \{ l \in \mathcal{L}|\, d^*_{lt} = 0\} $ and $\mathcal{G}_t^{\mathrm{0}} = \{ g \in \mathcal{G}|\, p^*_{gt} = 0\} $.

We define the maximum positive change in $\rho_t$ for $t\in\mathcal{X}^{\mathrm{d2}}$ as
\begin{align}
    \overline{\Delta}^+ = \min_{t\in\mathcal{X}^{\mathrm{d2}}} \big \{\overline{\chi}^*_{lt} \; | \, l \in \mathcal{L}_t^{\mathrm{+}} \big \} \cup  \big \{\underline{\mu}^*_{gt}  \; | \,g \in \mathcal{G}_t^{\mathrm{0}} \big \}  \,.
\end{align} 
This is defined so that an increase of $\rho_{H+1}$ by $\overline{\Delta}^+$ can be counterbalanced by changing $\overline{\mu}$, $\underline{\mu}$, $\overline{\chi}$, $\underline{\chi}$, so that KKT conditions corresponding to \eqref{eq:dual_gens} and \eqref{eq:dual_loads} still hold. Similarly, the maximum negative change in $\rho_t$ for $t\in\mathcal{X}^{\mathrm{d2}}$ is 
\begin{align}
    \overline{\Delta}^- = \min_{t\in\mathcal{X}^{\mathrm{d2}}} \big \{\underline{\chi}^*_{lt} \; | \, l \in \mathcal{L}_t^{\mathrm{0}} \big \} \cup  \big \{\overline{\mu}^*_{gt}  \; | \, g \in \mathcal{G}_t^{\mathrm{+}} \big \} \,.
\end{align}
Non-negativity of $\overline{\Delta}^+$ and $\overline{\Delta}^-$ follows from \eqref{eq:utilcon_new}-\eqref{eq:costcon_new_2} and \eqref{eq:costcon_new}-\eqref{eq:utilcon_new_2}, respectively.
In fact, at least one of $\overline{\Delta}^+$ and $\overline{\Delta}^-$ is strictly positive
\begin{align}\label{eq:deltaplus}
    \overline{\Delta}^+ > 0 \, \vee \, \overline{\Delta}^- > 0 \,,
\end{align}
since the inequalities in \eqref{eq:utilcon_new} and \eqref{eq:costcon_new} are strict for all but a single load or generator. 

Other than KKT conditions corresponding to \eqref{eq:dual_gens} and \eqref{eq:dual_loads}, the dual variable $\rho_t$ for $t\in\mathcal{X}^{\mathrm{d2}}$ is bound by an end of time-zone relation due to the KKT condition corresponding to dual constraint~\eqref{eq:dual_timelink}:
\begin{align}\label{eq:endofTZcondition}
    \rho_X^* = \rho_{X+1}^* + \underline{\nu}_X^* - \overline{\nu}_X^* \,,
\end{align}
where $X$ denotes the final time period in $\mathcal{X}$. 
As $X < T$, it holds that $\underline{\nu}_X^* > 0 \vee \overline{\nu}_X^* >0$ by the definition of a time zone. From \eqref{eq:endofTZcondition} we can derive the following additional constraint on $\delta$:
\begin{align}\label{eq:deltaend1}
    \delta &\in [-\infty, \overline{\nu}_X^*] && \text{if } \overline{\nu}_X^* > 0 \\
    \label{eq:deltaend2}
    \delta &\in [-\underline{\nu}_X^*, \infty] && \text{if } \underline{\nu}_X^* > 0 \,.
\end{align}
Combining \eqref{eq:deltaplus}, \eqref{eq:deltaend1}, and \eqref{eq:deltaend2},
\begin{align}\label{eq:deltatot}
    \delta \in &\left[-\overline{\Delta}^-, \min \big\{\overline{\nu}_X^*,\,  \overline{\Delta}^+\big\}\right]  && \text{if } \overline{\nu}_X^* > 0 \\
    \delta \in &\left[-\min\big\{\underline{\nu}_X^*, \, \overline{\Delta}^- \big\}, \overline{\Delta}^+\right]  && \text{if } \underline{\nu}_X^* > 0 \,.
\end{align}
Condition \eqref{eq:deltaplus} ensures that either the upper or lower bound on $\delta$ is nonzero, in both these cases. 

Finally, we can construct the weak dual solution accordingly:
\begin{align}
    &\rho_{t}' =  \rho_t^* + \delta & \forall t \in \mathcal{X}^{\mathrm{d2}} \\
    &\overline{\chi}_{lt}' = \overline{\chi}_{lt}^* - \delta & \forall t \in \mathcal{X}^{\mathrm{d2}}, l \in \mathcal{L}_t^{\mathrm{+}} \\
    &\underline{\chi}_{lt}' = \underline{\chi}_{lt}^* + \delta & \forall t \in \mathcal{X}^{\mathrm{d2}}, l \in \mathcal{L}_t^{\mathrm{0}} \\
    &\overline{\mu}_{gt}' = \overline{\mu}_{gt}^* + \delta & \forall t \in \mathcal{X}^{\mathrm{d2}}, g \in \mathcal{G}_t^{\mathrm{+}} \\
    &\underline{\mu}_{gt}' = \underline{\mu}_{gt}^* - \delta & \forall t \in \mathcal{X}^{\mathrm{d2}}, g \in \mathcal{G}_t^{\mathrm{0}} \\
    & \underline{\nu}_X' - \overline{\nu}_X'  = \underline{\nu}_X^* - \overline{\nu}_X^* + \delta \,.
\end{align}
\end{proof}
Intuitively, the conditions in Theorem \ref{thrm:main} describe the setting in which the full-horizon optimal solution has a time zone which extends over the two different days, but does not cover the entire second day. When this time zone is split over two market-clearing intervals, at least one of the parts will lack a marginal load or generator, and the optimal market-clearing price for the affected time periods can take multiple values. 
These conditions are \textit{sufficient} but not necessary, i.e. there are other cases in which a weak dual solution to the split-horizon problem exists.

Regarding the assumptions in Theorem \ref{thrm:main}, the uniqueness of cost and utility bids excludes certain miscellaneous exceptional cases, such as the case that all market participants are either on an upper or lower bound, but there is still no freedom in the related dual variables.
We further assume that the time zone $\mathcal{X}$ that spans the two days does not include $t=T$, the final time period of day two. 
This assumption is used to exclude the following very specific and unlikely case.
It could happen that $T\in\mathcal{X}$, while at the same time $e_T^* \in (\underline{S}, \overline{E})$. By KKT condition given by dual constraint \eqref{eq:dual_timelink_H}, this implies that $\rho_t^* = 0$ for all $t \in \mathcal{X}$. If there is a marginal generator bidding $0$ in the part of $\mathcal{X}$ that is on day one, this fixes $\rho'_H = 0$ for the split-horizon problem too.
Furthermore, due to dual constraint \eqref{eq:dual_timelink_H}, it must hold that $\rho_{H+1}'=0$.

\section{Restoring Market Efficiency} \label{sec:5}

In this section we propose a method to ensure equivalence between dual solutions to the split-horizon problem and the full-horizon problem. 
We explain how our method solves the problem in case of perfect foresight, and discuss its use in an imperfect foresight setting.

\subsection{Proposed Method} 
In the perfect foresight setting, the optimal dual variables to the full-horizon problem, including $\rho^*_{H+1}$, are assumed known or predicted perfectly when clearing the first-day market. 
We denote solutions to the restored split-horizon problem using $''$.
When clearing the market for the second day, $\rho_H''$, $\underline{\nu}_H''$, and $\overline{\nu}_H''$ are known. 

Our method aims to modify the split-horizon problem in order to restore the dual constraint that is missing in the split-horizon problem:
\begin{align}\label{eq:missing_dualrel}
    \rho_{H+1} = \rho_H - \underline{\nu}_H + \overline{\nu}_H \,,
\end{align}
while it is present in the full-horizon problem.

First, we modify problem $\textbf{C}(\mathcal{T}^{\mathrm{d1}})$. Instead of fixing the final storage level $e_H$ using a constraint, we use the value $\rho_{H+1}^*$ to steer the optimal dispatch of the storage. 
This is done by adding $e_H \, \rho_{H+1}^*$ to the objective, where $\rho_{H+1}^*$ is the perfect prediction of this dual variable. 
As a result, the first-day dual problem includes the dual constraint
\begin{align}\label{eq:restored_rel_d1}
        & - \underline{\nu}_H + \overline{\nu}_H + \rho_H - \rho_{H+1}^* = 0 
\end{align}
instead of \eqref{eq:dual_timelink_H}, 
where it should be noted that $\rho_{H+1}^*$ is a parameter in the day 1 problem.

Next, we modify problem $\textbf{F}(\mathcal{T}^{\mathrm{d2}})$, using $\rho_{H}''$, $\underline{\nu}_{H}''$, and $\overline{\nu}_{H}''$ determined by the modified day 1 problem as parameters.
The initial storage level $e_{H}$ is now included in the second day problem as a {variable} instead of a parameter. We add $-e_{H} \, (\rho_{H}'' - \underline{\nu}_{H}'' + \overline{\nu}_{H}'')$ to the objective function.
The dual problem for day 2 will then include the following constraint related to $e_{H}$:
\begin{align}\label{eq:restored_rel_d2}
    - \rho_{H+1} + \rho_{H}'' - \underline{\nu}_{H}'' + \overline{\nu}_{H}'' = 0    \,,
\end{align}
which equals the missing dual constraint, with dual variables related to day~1 as parameters. 

However, these modifications introduce the possibility to have multiple solutions in the primal problem. We can overcome this by reintroducing the constraint to fix the final storage level in the first day problem. In that case, the price for the last hour of the first day is the sum $\lambda_H + \xi$.
The dual variable $\xi$ needs to be included to the objective function for the second day too, such that the added term is $-e_{H} \, (\rho_{H}'' + \xi'' - \underline{\nu}_{H}'' + \overline{\nu}_{H}'')$.
We also include a constraint to fix the variable $e_{H}$ to be equal to $e_{H}^*$ in the problem of the second day, with an associated dual variable $\zeta$. In that case, the price for the first hour of the second day is the sum $\lambda_{H+1} + \zeta$. Note that a similar approach can be followed to restore missing dual constraints when splitting a problem with any type of time-linking constraint. In the context of market clearing, this could be, for example, ramp rates of generators.

The added term $e_H \, \rho_{H+1}^*$ in the objective of the first-day market resembles the `future value of stored water' commonly applied in optimal scheduling problems for hydropower. In that application, $\rho_{H+1}^*$ would represent the opportunity cost of each unit left in the storage at the end of the considered period. The future-value-of-storage function $f(e_H)=e_H \, \rho_{H+1}^*$ in our proposed solution only represents that value correctly at the optimal value of $e_H^*$. In an imperfect foresight setting, previous works estimated this function for all feasible values of $e_H$, for instance using dynamic programming and Benders decomposition \cite{pereira1991multi}. 
On the other hand, our proposed added term to the objective function of the second-day market is not seen in hydropower scheduling applications. 
This term represents the value of stored energy from past market-clearing intervals, which is essential for ensuring that the prices support the dispatch.

\subsection{Illustrative Example}
We apply the proposed method to {Illustrative Example II} from Section \ref{sec:illex}. In the perfect foresight case, we assume that the optimal value $\rho_2^* = 5 \si{\euro\per\watt\hour}$ is known at the time of market clearing for day 1. 
Using this value, we obtain $\rho''_H +\xi'' = 5 \si{\euro\per\watt\hour}$. 
The second day problem no longer admits multiple dual solutions, because the new dual constraint \eqref{eq:restored_rel_d2} enforces that $\rho_{H+1}'' +\zeta'' = \rho''_H +\xi'' = 5 \si{\euro\per\watt\hour}$. In summary, prices for both $H$ and $H+1$ are equal to their original values in the full-horizon problem. As a result, market efficiency (including dispatch-following incentives) and cost recovery for the storage are again ensured.

\subsection{Imperfect Foresight}
In an imperfect foresight setting, the market operator would make an error in estimating $\rho_{H+1}$ and $e_H$. The resulting market prices might be misleading. As we force the use of the storage to reach the specified storage end level, some loads could be asked to pay a price that is higher than their utility and some generators could get paid with a price that is lower than their costs. Market efficiency is thus not ensured, as well as cost recovery for loads and generators.


\section{Discussion \& Conclusion} \label{sec:conclusion}

The inclusion of non-merchant storage in energy market clearing has received attention, among others for its potential to increase social welfare. 
In this work, we have argued against several simplifying assumptions that are commonly made in the literature regarding the final state of energy of the non-merchant storage, in particular, to set the final state of energy equal to the initial state of energy, or to disregard the final level altogether. We have shown that under these assumptions, market efficiency is not ensured, as the market may fail to provide dispatch-following incentives for the storage system. In addition, the storage can only perform time arbitrage within a single market interval, but not between market horizons, resulting in a loss of overall social welfare. 
However, allowing the final state of energy to take any value can also have negative consequences on the market properties if not handled carefully. 
First, one must determine the value of the final state of energy that is optimal for the storage system.
Second, the market prices may not reflect the relation between different market horizons.

Regarding the latter, we have shown that the split-horizon market may fail to provide dispatch-following incentives for the storage, even when the final state of energy is set perfectly. 
However, we have shown that any solution to the full-horizon problem is also a solution to the split-horizon problem, in the perfect-foresight setting. This changes the nature of the problem, compared to what is discussed in the literature. It shows that there always \textit{exists} a property-preserving dual solution to the split-horizon problem, namely, the dual solution to the full-horizon problem. However, there may in addition exist optimal dual solutions to the split-horizon problem, which are infeasible to the full-horizon dual problem and lead to market inefficiencies. Therefore, the challenge \textit{in the perfect foresight setting} becomes that of selecting the correct dual solution. We have provided sufficient conditions for the existence of weak dual solutions. 
Finally, we have proposed a method that restores market properties in the split-horizon problem, in the perfect-foresight case.

However, it becomes more complicated in the imperfect-foresight case, where the value of $E^{\mathrm{end}}$ may turn out to be suboptimal in hindsight. The proposed solution cannot be applied in case $e_H$ is suboptimal in hindsight. Future work should develop solutions for the imperfect-foresight case. Here, it should be considered that the suboptimal final state of energy leads to a social welfare loss compared to the perfect-foresight case. It is a nontrivial question how this loss should be fairly divided among the market participants, especially since the storage may improve social welfare in expectation. 

Furthermore, we note that illustrative examples and proofs in this work are valid for a market clearing with linear cost and utility functions. We have restricted ourselves to this setting for simplicity, and because this type of market clearing is common in practice.
However, the market efficiency problem can also arise in the nonlinear convex case, under similar conditions. Future works should analyze this case in further detail. 

The negative effects of simplifying assumptions on the final level of the storage are most critical for very large storage systems, as the seasonal arbitrage that these assets could provide would not be harnessed. Future work should assess for which storage characteristics these simplifying assumptions have the greatest negative impact.
Finally, we have neglected the problem of determining the value for the final state of energy. Future work should focus on how to determine this value, both under perfect and imperfect information.


\begin{appendices}
\section{Nomenclature and Notations} \label{ap:nomencl}
\printnomenclature

\nomenclature[C]{$e_t$}{State of energy at the end of $t$ (\si{{\watt\hour}})}
\nomenclature[C]{$b_t$}{Energy charged or discharged at $t$ (\si{{\watt\hour}})}
\nomenclature[C]{$\mathbf{x}$}{Vector grouping all the primal variables}
\nomenclature[C]{$p_{gt}$}{Production of $g$ at $t$ (\si{{\watt\hour}})}
\nomenclature[C]{$d_{lt}$}{Demand of $l$ at $t$ (\si{{\watt\hour}})}

\nomenclature[A]{$g \in \mathcal{G}$}{Generators}
\nomenclature[A]{$l \in \mathcal{L}$}{Loads}
\nomenclature[A]{$t \in \mathcal{T}$}{Time periods}

\nomenclature[B]{$\overline{E}$}{Storage capacity (\si{{\watt\hour}})}
\nomenclature[B]{$E^{\mathrm{init}}$}{Initial state of energy (\si{{\watt\hour}})}
\nomenclature[B]{$E^{\mathrm{end}}$}{Final state of energy (\si{{\watt\hour}})}
\nomenclature[B]{$U_{lt}$}{Utility of $l$ at $t$ (\si{\euro\per\watt\hour})}
\nomenclature[B]{$C_{gt}$}{Cost of $g$ at $t$ (\si{\euro\per\watt\hour})}
\nomenclature[B]{$\overline{P}_{gt}$}{Maximum output of $g$ at $t$ (\si{{\watt\hour}})}
\nomenclature[B]{$\overline{D}_{lt}$}{Maximum consumption of $l$ at $t$ (\si{{\watt\hour}})}

\nomenclature[D]{$\lambda_t$}{Dual variable associated with the energy balance at time $t$, representing the market price in that time period (\si{\euro\per\watt\hour})}
\nomenclature[D]{$\boldsymbol{\zeta}$}{Vector grouping all the dual variables}
\subsection*{Notation Rules}

Sets are represented with calligraphic capital letters, e.g., $\mathcal{T}$. Primal variables are represented with lower-case Roman letters ($p_{gt}$ for instance), while dual variables are represented with Greek letters (for example, $\xi$). Characters in bold font indicate vectors grouping all the variables. Parameters are given with capital letters (e.g., $C_{gt}$). The overline is used for maximum and the underline for minimum, $\overline{P}$ and $\underline{P}$, respectively.

Indices are given as subscript, as for the example of $._t$ for time indices. Superscripts are used to further describe the variables or parameters (e.g., $E^{\mathrm{end}}$), while the superscript $*$ is used to describe an optimal value of primal or dual variables. The superscript $'$ is used to describe an optimal value of primal or dual variables of the split-horizon problem (in Section \ref{sec:4}, e.g., $\lambda_t'$). Finally, the superscript $''$ is used to describe an optimal value of primal or dual variables for the restored split-horizon problem (in Section \ref{sec:5}, e.g., $\rho_H''$).

\section{KKT Conditions}\label{ap:KKT}
\subsection{KKT Conditions Full-Horizon Problem}\label{ap:fullKKT}
The KKT conditions of $\text{\textbf{N}}(\mathcal{T})$ are
\begin{subequations}\label{prob:simplified_dispatch_KKTs}
\allowdisplaybreaks
\begin{align}
    \label{eq:KKT_1a}
    & C_{gt} - \underline{\mu}_{gt} + \overline{\mu}_{gt} - \lambda_t = 0, & \forall g \in \mathcal{G}, t \in \mathcal{T}  \\
    \label{eq:KKT_1b}
    & -U_{lt} - \underline{\chi}_{lt} + \overline{\chi}_{lt} + \lambda_t = 0, & \forall l \in \mathcal{L}, t \in \mathcal{T}  \\
    \label{eq:KKT_2}
    & - \rho_t + \lambda_t = 0, & \forall t \in \mathcal{T} \\
    \label{eq:KKT_3a}
    &-\underline{\nu}_t + \overline{\nu}_t +\rho_t - \rho_{t+1} = 0, & \forall t \in\mathcal{T}\setminus \{T\}\\
    \label{eq:KKT_3b}
    &-\underline{\nu}_T + \overline{\nu}_T +\rho_T = 0 \\
    \label{eq:KKT_4} 
    & 0  \leq p_{gt} \perp \underline{\mu}_{gt} \geq 0 , & \forall g \in \mathcal{G}, t \in \mathcal{T}\\
    \label{eq:KKT_5} 
    &0 \leq \overline{P}_{gt} -p_{gt} \perp \overline{\mu}_{gt} \geq 0 , & \forall g \in \mathcal{G}, t \in \mathcal{T}\\
    \label{eq:KKT_6} 
    & 0 \leq d_{lt} \perp \underline{\chi}_{lt} \geq 0 , & \forall l \in \mathcal{L}, t \in \mathcal{T} \\
    \label{eq:KKT_7} 
    & 0 \leq  \overline{D}_{lt}-d_{lt} \perp \overline{\chi}_{lt} \geq 0 , & \forall l \in \mathcal{L}, t \in \mathcal{T}\\
    \label{eq:KKT_8} 
    & 0 \leq e_t  \perp \underline{\nu}_{t} \geq 0 , & \forall t \in \mathcal{T} \\
    \label{eq:KKT_9} 
    & 0 \leq  \overline{E}-e_t \perp \overline{\nu}_{t} \geq 0 , & \forall t \in \mathcal{T} \\
    \label{eq:KKT_10} 
    & \eqref{eq:energy_bal} \;, \eqref{eq:stor_bal_t} \;, \eqref{eq:stor_bal_1} .
\end{align}
\end{subequations}

\subsection{KKT Conditions Split-Horizon Problem}\label{ap:splitKKT}
The KKT conditions of $\text{\textbf{L}}(\mathcal{T}^{\mathrm{d1}})$ are
\begin{subequations}\label{prob:simplified_dispatch_KKTs_D1}
\allowdisplaybreaks
\begin{align}
    \label{eq:KKT_1a_d1}
    & C_{gt} - \underline{\mu}_{gt} + \overline{\mu}_{gt} - \lambda_t = 0, & \forall g \in \mathcal{G}, t \in \mathcal{T}^{\mathrm{d1}}  \\
    \label{eq:KKT_1b_d1}
    & -U_{lt} - \underline{\chi}_{lt} + \overline{\chi}_{lt} + \lambda_t = 0, & \forall l \in \mathcal{L}, t \in \mathcal{T}^{\mathrm{d1}} \\
    \label{eq:KKT_2_d1}
    & - \rho_t + \lambda_t = 0, & \forall t \in \mathcal{T}^{\mathrm{d1}} \\
    \label{eq:KKT_3a_d1}
    &-\underline{\nu}_t + \overline{\nu}_t +\rho_t - \rho_{t+1} = 0, & \forall t \in\mathcal{T}^{\mathrm{d1}}\setminus \{H\}\\
    \label{eq:KKT_3b_d1}
    &-\underline{\nu}_H + \overline{\nu}_H +\rho_H - \xi = 0\\
    \label{eq:KKT_4_d1} 
    & 0  \leq p_{gt} \perp \underline{\mu}_{gt} \geq 0 , & \forall g \in \mathcal{G}, t \in \mathcal{T}^{\mathrm{d1}}\\
    \label{eq:KKT_5_d1} 
    &0 \leq \overline{P}_{gt} -p_{gt} \perp \overline{\mu}_{gt} \geq 0 , & \forall g \in \mathcal{G}, t \in \mathcal{T}^{\mathrm{d1}}\\
    \label{eq:KKT_6_d1} 
    & 0 \leq d_{lt} \perp \underline{\chi}_{lt} \geq 0 , & \forall l \in \mathcal{L}, t \in \mathcal{T}^{\mathrm{d1}} \\
    \label{eq:KKT_7_d1} 
    & 0 \leq  \overline{D}_{lt}-d_{lt} \perp \overline{\chi}_{lt} \geq 0 , & \forall l \in \mathcal{L}, t \in \mathcal{T}^{\mathrm{d1}}\\
    \label{eq:KKT_8_d1} 
    & 0 \leq e_t  \perp \underline{\nu}_{t} \geq 0 , & \forall t \in \mathcal{T}^{\mathrm{d1}} \\
    \label{eq:KKT_9_d1} 
    & 0 \leq  \overline{E}-e_t \perp \overline{\nu}_{t} \geq 0 , & \forall t \in \mathcal{T}^{\mathrm{d1}} \\
    \label{eq:KKT_9b_d1} 
    &  e_H = {E}^{\mathrm{end}} \\ 
    & e_1 = E^{\mathrm{init}} + b_{1} \\
    & e_t = e_{t-1} + b_{t}, & \forall t \in\mathcal{T}^{\mathrm{d1}}\setminus\{1\} \\
    & \sum_{l\in\mathcal{L}}d_{lt} + b_{t} -\sum_{g\in\mathcal{G}}p_{gt}= 0 , & \forall t \in \mathcal{T}^{\mathrm{d1}} .
\end{align}
\end{subequations}

The KKT conditions of $\text{\textbf{N}}(\mathcal{T}^{\mathrm{d2}})$ are
\begin{subequations}\label{prob:simplified_dispatch_KKTs_D2}
\allowdisplaybreaks
\begin{align}
    \label{eq:KKT_1a_d2}
    & C_{gt} - \underline{\mu}_{gt} + \overline{\mu}_{gt} - \lambda_t = 0, & \forall g \in \mathcal{G}, t \in \mathcal{T}^{\mathrm{d2}}  \\
    \label{eq:KKT_1b_d2}
    & -U_{lt} - \underline{\chi}_{lt} + \overline{\chi}_{lt} + \lambda_t = 0, & \forall l \in \mathcal{L}, t \in \mathcal{T}^{\mathrm{d2}}  \\
    \label{eq:KKT_2_d2}
    & - \rho_t + \lambda_t = 0, & \forall t \in \mathcal{T}^{\mathrm{d2}} & \\
    \label{eq:KKT_3a_d2}
    &-\underline{\nu}_t + \overline{\nu}_t +\rho_t - \rho_{t+1} = 0, & \forall t \in\mathcal{T}^{\mathrm{d2}}\setminus \{T\} \\
    \label{eq:KKT_3b_d2}
    &-\underline{\nu}_T + \overline{\nu}_T +\rho_T = 0 \\
    \label{eq:KKT_4_d2} 
    & 0  \leq p_{gt} \perp \underline{\mu}_{gt} \geq 0 , & \forall g \in \mathcal{G}, t \in \mathcal{T}^{\mathrm{d2}}\\
    \label{eq:KKT_5_d2} 
    &0 \leq \overline{P}_{gt} -p_{gt} \perp \overline{\mu}_{gt} \geq 0 , & \forall g \in \mathcal{G}, t \in \mathcal{T}^{\mathrm{d2}}\\
    \label{eq:KKT_6_d2} 
    & 0 \leq d_{lt} \perp \underline{\chi}_{lt} \geq 0 , & \forall l \in \mathcal{L}, t \in \mathcal{T}^{\mathrm{d2}} \\
    \label{eq:KKT_7_d2} 
    & 0 \leq  \overline{D}_{lt}-d_{lt} \perp \overline{\chi}_{lt} \geq 0 , & \forall l \in \mathcal{L}, t \in \mathcal{T}^{\mathrm{d2}}\\
    \label{eq:KKT_8_d2} 
    & 0 \leq e_t  \perp \underline{\nu}_{t} \geq 0 , & \forall t \in \mathcal{T}^{\mathrm{d2}} \\
    \label{eq:KKT_9_d2} 
    & 0 \leq  \overline{E}-e_t \perp \overline{\nu}_{t} \geq 0 , & \forall t \in \mathcal{T}^{\mathrm{d2}}\\
    & e_{H+1} = E^{\mathrm{init}} + b_{H+1} \\
    & e_t = e_{t-1} + b_{t}, & \forall t \in\mathcal{T}^{\mathrm{d2}} \setminus \{H+1\} \\
    & \sum_{l\in\mathcal{L}}d_{lt} + b_{t} -\sum_{g\in\mathcal{G}}p_{gt}= 0 , & \forall t \in \mathcal{T}^{\mathrm{d2}} .
\end{align}
\end{subequations}

\section{Proofs}

\subsection{Proof of Lemma \ref{lem:equivprimal}}\label{ap:equivprimal}
\begin{proof}
The full-horizon and split-horizon problems have the same objective function. 
The optimal solution to the full-horizon problem $\mathbf{x}^*$ lies in the feasible space $\mathcal{F}'$ of the split-horizon problem, i.e., $\mathbf{x}^* \in \mathcal{F}'$, because it satisfies the additional constraint $e_H = E^{\mathrm{end}} = e_H^*$. Therefore, the optimal objective of the full-horizon problem is a lower bound for the split-horizon problem. 
The feasible space $\mathcal{F}'$ of the split-horizon problem with the added constraint is a subspace of the feasible space $\mathcal{F}$ of the full-horizon problem, i.e., $\mathcal{F}' \subset \mathcal{F}$. Thus, the optimal objective of the split-horizon problem is a lower bound for the full-horizon problem. This implies that the split- and full-horizon problems attain the same maximum. 

As any optimum $\mathbf{x}' \in \mathcal{F}$ and any optimum $\mathbf{x}^* \in \mathcal{F}'$, the sets of optimal primal solutions must be equal.

\end{proof}

\subsection{Proof of Lemma \ref{lem:admitdual}}\label{ap:admitdual}
\begin{proof}
As the problem is convex, the KKT conditions are necessary and sufficient. Therefore, any feasible point in the KKT conditions of the split-horizon problem is an optimal solution to it. 
The KKT conditions of the split-horizon problem only differ from those of the full-horizon problem by equations:
\begin{align}
    \label{eq:KKT_replace}
    &-\underline{\nu}_H + \overline{\nu}_H +\rho_H - \xi = 0 \\
    &e_H = {E}^{\mathrm{end}} \label{eq:fixed_lem}\,,
\end{align}
which replace the full-horizon KKT condition 
\begin{align}
    &-\underline{\nu}_H + \overline{\nu}_H +\rho_H - \rho_{H+1} = 0 \,. 
\end{align}
The assumption of perfect foresight implies that $E^{\mathrm{end}}  = e^*_H$. Therefore, KKT condition \eqref{eq:fixed_lem} holds for any optimal primal solution to the full-horizon problem. 
As for condition \eqref{eq:KKT_replace}, we can see that it is satisfied by the optimal solution to the full-horizon problem in case $\xi' = \rho_{H+1}^*$.

Therefore, any optimal solution to the full-horizon problem, augmented with $\xi' = \rho_{H+1}^*$, also satisfies the KKT conditions for the split-horizon problem. As the KKT conditions are sufficient for convex problems, this means that this solution is also an optimal solution to the split-horizon problem. 
\end{proof}
\end{appendices}

\section*{Declarations}

\paragraph{Funding} The work of L. Fr\"{o}lke, P. Pinson, and J. Kazempour was supported in part by the EMB3Rs project, funded by the European Commission Horizon 2020 program, Grant Agreement No. 847121.

\paragraph{Competing interests} The authors have no competing interests to declare that are relevant to the content of this article.


\bibliography{references}


\begin{thebibliography}{28}
\ifx \bisbn   \undefined \def \bisbn  #1{ISBN #1}\fi
\ifx \binits  \undefined \def \binits#1{#1}\fi
\ifx \bauthor  \undefined \def \bauthor#1{#1}\fi
\ifx \batitle  \undefined \def \batitle#1{#1}\fi
\ifx \bjtitle  \undefined \def \bjtitle#1{#1}\fi
\ifx \bvolume  \undefined \def \bvolume#1{\textbf{#1}}\fi
\ifx \byear  \undefined \def \byear#1{#1}\fi
\ifx \bissue  \undefined \def \bissue#1{#1}\fi
\ifx \bfpage  \undefined \def \bfpage#1{#1}\fi
\ifx \blpage  \undefined \def \blpage #1{#1}\fi
\ifx \burl  \undefined \def \burl#1{\textsf{#1}}\fi
\ifx \doiurl  \undefined \def \doiurl#1{\url{https://doi.org/#1}}\fi
\ifx \betal  \undefined \def \betal{\textit{et al.}}\fi
\ifx \binstitute  \undefined \def \binstitute#1{#1}\fi
\ifx \binstitutionaled  \undefined \def \binstitutionaled#1{#1}\fi
\ifx \bctitle  \undefined \def \bctitle#1{#1}\fi
\ifx \beditor  \undefined \def \beditor#1{#1}\fi
\ifx \bpublisher  \undefined \def \bpublisher#1{#1}\fi
\ifx \bbtitle  \undefined \def \bbtitle#1{#1}\fi
\ifx \bedition  \undefined \def \bedition#1{#1}\fi
\ifx \bseriesno  \undefined \def \bseriesno#1{#1}\fi
\ifx \blocation  \undefined \def \blocation#1{#1}\fi
\ifx \bsertitle  \undefined \def \bsertitle#1{#1}\fi
\ifx \bsnm \undefined \def \bsnm#1{#1}\fi
\ifx \bsuffix \undefined \def \bsuffix#1{#1}\fi
\ifx \bparticle \undefined \def \bparticle#1{#1}\fi
\ifx \barticle \undefined \def \barticle#1{#1}\fi
\bibcommenthead
\ifx \bconfdate \undefined \def \bconfdate #1{#1}\fi
\ifx \botherref \undefined \def \botherref #1{#1}\fi
\ifx \url \undefined \def \url#1{\textsf{#1}}\fi
\ifx \bchapter \undefined \def \bchapter#1{#1}\fi
\ifx \bbook \undefined \def \bbook#1{#1}\fi
\ifx \bcomment \undefined \def \bcomment#1{#1}\fi
\ifx \oauthor \undefined \def \oauthor#1{#1}\fi
\ifx \citeauthoryear \undefined \def \citeauthoryear#1{#1}\fi
\ifx \endbibitem  \undefined \def \endbibitem {}\fi
\ifx \bconflocation  \undefined \def \bconflocation#1{#1}\fi
\ifx \arxivurl  \undefined \def \arxivurl#1{\textsf{#1}}\fi
\csname PreBibitemsHook\endcsname

\bibitem[\protect\citeauthoryear{IEA}{2021}]{iea2021world}
\begin{botherref}
\oauthor{\bsnm{IEA}}:
{World Energy Outlook}.
Accessed 06-07-2022
(2021).
\url{iea.org/reports/world-energy-outlook-2021}
\end{botherref}
\endbibitem

\bibitem[\protect\citeauthoryear{FERC}{2108}]{FERC2018}
\begin{botherref}
\oauthor{\bsnm{FERC}}:
{Docket Nos. RM16-23-000 and AD16-20-000, Order No. 841: Electric Storage
  Participation in Markets Operated by Regional Transmission Organizations and
  Independent System Operators}.
Accessed 27-09-2022
(2108).
\url{www.ferc.gov/media/order-no-841}
\end{botherref}
\endbibitem

\bibitem[\protect\citeauthoryear{Singhal and Ela}{2020}]{Singhal2020Pricing}
\begin{bchapter}
\bauthor{\bsnm{Singhal}, \binits{N.G.}},
\bauthor{\bsnm{Ela}, \binits{E.G.}}:
\bctitle{Pricing impacts of state of charge management options for electric
  storage resources}.
In: \bbtitle{2020 IEEE Power \& Energy Society General Meeting (PESGM)},
pp. \bfpage{1}--\blpage{6}
(\byear{2020}).
\doiurl{10.1109/PESGM41954.2020.9282160}
\end{bchapter}
\endbibitem

\bibitem[\protect\citeauthoryear{Hartwig and Kockar}{2016}]{Hartwig2016Impact}
\begin{barticle}
\bauthor{\bsnm{Hartwig}, \binits{K.}},
\bauthor{\bsnm{Kockar}, \binits{I.}}:
\batitle{Impact of strategic behavior and ownership of energy storage on
  provision of flexibility}.
\bjtitle{IEEE Transactions on Sustainable Energy}
\bvolume{7}(\bissue{2}),
\bfpage{744}--\blpage{754}
(\byear{2016})
\doiurl{10.1109/TSTE.2015.2497967}
\end{barticle}
\endbibitem

\bibitem[\protect\citeauthoryear{Sioshansi}{2014}]{Sioshansi2014When}
\begin{barticle}
\bauthor{\bsnm{Sioshansi}, \binits{R.}}:
\batitle{When energy storage reduces social welfare}.
\bjtitle{Energy Economics}
\bvolume{41},
\bfpage{106}--\blpage{116}
(\byear{2014})
\doiurl{10.1016/j.eneco.2013.09.027}
\end{barticle}
\endbibitem

\bibitem[\protect\citeauthoryear{Gu and Sioshansi}{2022}]{Gu2022Market}
\begin{barticle}
\bauthor{\bsnm{Gu}, \binits{W.}},
\bauthor{\bsnm{Sioshansi}, \binits{R.}}:
\batitle{Market equilibria with energy storage as flexibility resources}.
\bjtitle{IEEE Open Access Journal of Power and Energy}
\bvolume{9},
\bfpage{584}--\blpage{597}
(\byear{2022})
\doiurl{10.1109/OAJPE.2022.3217973}
\end{barticle}
\endbibitem

\bibitem[\protect\citeauthoryear{Taylor}{2014}]{Taylor2014Financial}
\begin{barticle}
\bauthor{\bsnm{Taylor}, \binits{J.A.}}:
\batitle{Financial storage rights}.
\bjtitle{IEEE Transactions on Power Systems}
\bvolume{30}(\bissue{2}),
\bfpage{997}--\blpage{1005}
(\byear{2014})
\doiurl{10.1109/TPWRS.2014.2339016}
\end{barticle}
\endbibitem

\bibitem[\protect\citeauthoryear{Mu{\~n}oz-{\'A}lvarez and
  Bitar}{2017}]{Munoz2017Financial}
\begin{barticle}
\bauthor{\bsnm{Mu{\~n}oz-{\'A}lvarez}, \binits{D.}},
\bauthor{\bsnm{Bitar}, \binits{E.}}:
\batitle{Financial storage rights in electric power networks}.
\bjtitle{Journal of Regulatory Economics}
\bvolume{52}(\bissue{1}),
\bfpage{1}--\blpage{23}
(\byear{2017})
\doiurl{10.1007/s11149-017-9327-9}
\end{barticle}
\endbibitem

\bibitem[\protect\citeauthoryear{Weibelzahl and
  M{\"a}rtz}{2018}]{weibelzahl2018effects}
\begin{barticle}
\bauthor{\bsnm{Weibelzahl}, \binits{M.}},
\bauthor{\bsnm{M{\"a}rtz}, \binits{A.}}:
\batitle{On the effects of storage facilities on optimal zonal pricing in
  electricity markets}.
\bjtitle{Energy Policy}
\bvolume{113},
\bfpage{778}--\blpage{794}
(\byear{2018})
\doiurl{10.1016/j.enpol.2017.11.018}
\end{barticle}
\endbibitem

\bibitem[\protect\citeauthoryear{Jiang and Sioshansi}{2023}]{jiang2023duality}
\begin{barticle}
\bauthor{\bsnm{Jiang}, \binits{Y.}},
\bauthor{\bsnm{Sioshansi}, \binits{R.}}:
\batitle{What duality theory tells us about giving market operators the
  authority to dispatch energy storage}.
\bjtitle{The Energy Journal}
\bvolume{44}(\bissue{3}),
\bfpage{1}--\blpage{19}
(\byear{2023})
\doiurl{10.5547/01956574.44.2.yjia}
\end{barticle}
\endbibitem

\bibitem[\protect\citeauthoryear{Salas and
  Powell}{2018}]{Salas2018Benchmarking}
\begin{barticle}
\bauthor{\bsnm{Salas}, \binits{D.F.}},
\bauthor{\bsnm{Powell}, \binits{W.B.}}:
\batitle{Benchmarking a scalable approximate dynamic programming algorithm for
  stochastic control of grid-level energy storage}.
\bjtitle{INFORMS Journal on Computing}
\bvolume{30}(\bissue{1}),
\bfpage{106}--\blpage{123}
(\byear{2018})
\doiurl{10.1287/ijoc.2017.0768}
\end{barticle}
\endbibitem

\bibitem[\protect\citeauthoryear{Löhndorf and
  Minner}{2010}]{Loehndorf2010Optimal}
\begin{barticle}
\bauthor{\bsnm{Löhndorf}, \binits{N.}},
\bauthor{\bsnm{Minner}, \binits{S.}}:
\batitle{Optimal day-ahead trading and storage of renewable energies--an
  approximate dynamic programming approach}.
\bjtitle{Energy Systems}
\bvolume{1}(\bissue{1}),
\bfpage{61}--\blpage{77}
(\byear{2010})
\doiurl{10.1007/s12667-009-0007-4}
\end{barticle}
\endbibitem

\bibitem[\protect\citeauthoryear{Wang et~al.}{2017}]{Wang2017LookAhead}
\begin{barticle}
\bauthor{\bsnm{Wang}, \binits{Y.}},
\bauthor{\bsnm{Dvorkin}, \binits{Y.}},
\bauthor{\bsnm{Fernández-Blanco}, \binits{R.}},
\bauthor{\bsnm{Xu}, \binits{B.}},
\bauthor{\bsnm{Qiu}, \binits{T.}},
\bauthor{\bsnm{Kirschen}, \binits{D.S.}}:
\batitle{Look-ahead bidding strategy for energy storage}.
\bjtitle{IEEE Transactions on Sustainable Energy}
\bvolume{8}(\bissue{3}),
\bfpage{1106}--\blpage{1117}
(\byear{2017})
\doiurl{10.1109/TSTE.2017.2656800}
\end{barticle}
\endbibitem

\bibitem[\protect\citeauthoryear{Mohsenian-Rad}{2016}]{Mohsenianrad2016Coordinated}
\begin{barticle}
\bauthor{\bsnm{Mohsenian-Rad}, \binits{H.}}:
\batitle{Coordinated price-maker operation of large energy storage units in
  nodal energy markets}.
\bjtitle{IEEE Transactions on Power Systems}
\bvolume{31}(\bissue{1}),
\bfpage{786}--\blpage{797}
(\byear{2016})
\doiurl{10.1109/TPWRS.2015.2411556}
\end{barticle}
\endbibitem

\bibitem[\protect\citeauthoryear{Ding et~al.}{2017}]{Ding2017Optimal}
\begin{barticle}
\bauthor{\bsnm{Ding}, \binits{H.}},
\bauthor{\bsnm{Pinson}, \binits{P.}},
\bauthor{\bsnm{Hu}, \binits{Z.}},
\bauthor{\bsnm{Wang}, \binits{J.}},
\bauthor{\bsnm{Song}, \binits{Y.}}:
\batitle{Optimal offering and operating strategy for a large wind-storage
  system as a price maker}.
\bjtitle{IEEE Transactions on Power Systems}
\bvolume{32}(\bissue{6}),
\bfpage{4904}--\blpage{4913}
(\byear{2017})
\doiurl{10.1109/TPWRS.2017.2681720}
\end{barticle}
\endbibitem

\bibitem[\protect\citeauthoryear{Badoual and Moura}{2021}]{Badoual2021Learning}
\begin{botherref}
\oauthor{\bsnm{Badoual}, \binits{M.D.}},
\oauthor{\bsnm{Moura}, \binits{S.J.}}:
A learning-based optimal market bidding strategy for price-maker energy
  storage.
Proceedings of the American Control Conference,
526--532
(2021)
\doiurl{10.23919/ACC50511.2021.9483213}
\end{botherref}
\endbibitem

\bibitem[\protect\citeauthoryear{Zhao et~al.}{2020}]{Zhao2020}
\begin{barticle}
\bauthor{\bsnm{Zhao}, \binits{J.}},
\bauthor{\bsnm{Zheng}, \binits{T.}},
\bauthor{\bsnm{Litvinov}, \binits{E.}}:
\batitle{A multi-period market design for markets with intertemporal
  constraints}.
\bjtitle{IEEE Transactions on Power Systems}
\bvolume{35}(\bissue{4}),
\bfpage{3015}--\blpage{3025}
(\byear{2020})
\doiurl{10.1109/TPWRS.2019.2963022}
\end{barticle}
\endbibitem

\bibitem[\protect\citeauthoryear{Sioshansi et~al.}{2021}]{Sioshansi2021Energy}
\begin{barticle}
\bauthor{\bsnm{Sioshansi}, \binits{R.}},
\bauthor{\bsnm{Denholm}, \binits{P.}},
\bauthor{\bsnm{Arteaga}, \binits{J.}},
\bauthor{\bsnm{Awara}, \binits{S.}},
\bauthor{\bsnm{Bhattacharjee}, \binits{S.}},
\bauthor{\bsnm{Botterud}, \binits{A.}},
\bauthor{\bsnm{Cole}, \binits{W.}},
\bauthor{\bsnm{Cortes}, \binits{A.}},
\bauthor{\bsnm{De~Queiroz}, \binits{A.R.}},
\bauthor{\bsnm{Decarolis}, \binits{J.}},
\bauthor{\bsnm{Ding}, \binits{Z.}},
\bauthor{\bsnm{Diorio}, \binits{N.}},
\bauthor{\bsnm{Dvorkin}, \binits{Y.}},
\bauthor{\bsnm{Helman}, \binits{U.}},
\bauthor{\bsnm{Johnson}, \binits{J.}},
\bauthor{\bsnm{Konstantelos}, \binits{I.}},
\bauthor{\bsnm{Mai}, \binits{T.}},
\bauthor{\bsnm{Pandzic}, \binits{H.}},
\bauthor{\bsnm{Sodano}, \binits{D.}},
\bauthor{\bsnm{Stephen}, \binits{G.}},
\bauthor{\bsnm{Svoboda}, \binits{A.}},
\bauthor{\bsnm{Zareipour}, \binits{H.}},
\bauthor{\bsnm{Zhang}, \binits{Z.}}:
\batitle{Energy-storage modeling: State-of-the-art and future research
  directions}.
\bjtitle{IEEE Transactions on Power Systems}
\bvolume{37}(\bissue{2}),
\bfpage{860}--\blpage{875}
(\byear{2021})
\doiurl{10.1109/TPWRS.2021.3104768}
\end{barticle}
\endbibitem

\bibitem[\protect\citeauthoryear{Crespo Del~Granado
  et~al.}{2016}]{Crespo2016Impact}
\begin{barticle}
\bauthor{\bsnm{Crespo Del~Granado}, \binits{P.}},
\bauthor{\bsnm{Wallace}, \binits{S.W.}},
\bauthor{\bsnm{Pang}, \binits{Z.}}:
\batitle{The impact of wind uncertainty on the strategic valuation of
  distributed electricity storage}.
\bjtitle{Computational Management Science}
\bvolume{13}(\bissue{1}),
\bfpage{5}--\blpage{27}
(\byear{2016})
\doiurl{10.1007/s10287-015-0235-0}
\end{barticle}
\endbibitem

\bibitem[\protect\citeauthoryear{Zhang et~al.}{2020}]{Zhang2020Energy}
\begin{botherref}
\oauthor{\bsnm{Zhang}, \binits{J.}},
\oauthor{\bsnm{Gu}, \binits{N.}},
\oauthor{\bsnm{Wu}, \binits{C.}}:
Energy storage as public asset.
Proceedings of the Eleventh ACM International Conference on Future Energy
  Systems,
374--385
(2020)
\doiurl{10.1145/3396851.3397760}
\end{botherref}
\endbibitem

\bibitem[\protect\citeauthoryear{Vespermann
  et~al.}{2021}]{Vespermann2021Access}
\begin{barticle}
\bauthor{\bsnm{Vespermann}, \binits{N.}},
\bauthor{\bsnm{Hamacher}, \binits{T.}},
\bauthor{\bsnm{Kazempour}, \binits{J.}}:
\batitle{Access economy for storage in energy communities}.
\bjtitle{IEEE Transactions on Power Systems}
\bvolume{36}(\bissue{3}),
\bfpage{2234}--\blpage{2250}
(\byear{2021})
\doiurl{10.1109/TPWRS.2020.3033999}
\end{barticle}
\endbibitem

\bibitem[\protect\citeauthoryear{Hua et~al.}{2019}]{Hua2019Pricing}
\begin{barticle}
\bauthor{\bsnm{Hua}, \binits{B.}},
\bauthor{\bsnm{Schiro}, \binits{D.A.}},
\bauthor{\bsnm{Zheng}, \binits{T.}},
\bauthor{\bsnm{Baldick}, \binits{R.}},
\bauthor{\bsnm{Litvinov}, \binits{E.}}:
\batitle{Pricing in multi-interval real-time markets}.
\bjtitle{IEEE Transactions on Power Systems}
\bvolume{34}(\bissue{4}),
\bfpage{2696}--\blpage{2705}
(\byear{2019})
\doiurl{10.1109/TPWRS.2019.2891541}
\end{barticle}
\endbibitem

\bibitem[\protect\citeauthoryear{Chen et~al.}{2021}]{Chen2021Pricing}
\begin{bchapter}
\bauthor{\bsnm{Chen}, \binits{C.}},
\bauthor{\bsnm{Tong}, \binits{L.}},
\bauthor{\bsnm{Guo}, \binits{Y.}}:
\bctitle{Pricing energy storage in real-time market}.
In: \bbtitle{2021 IEEE Power \& Energy Society General Meeting (PESGM)},
pp. \bfpage{1}--\blpage{5}
(\byear{2021}).
\doiurl{10.1109/PESGM46819.2021.9638013} .
\bcomment{IEEE}
\end{bchapter}
\endbibitem

\bibitem[\protect\citeauthoryear{Guo et~al.}{2021}]{Guo2021Pricing}
\begin{barticle}
\bauthor{\bsnm{Guo}, \binits{Y.}},
\bauthor{\bsnm{Chen}, \binits{C.}},
\bauthor{\bsnm{Tong}, \binits{L.}}:
\batitle{Pricing multi-interval dispatch under uncertainty part {I}:
  Dispatch-following incentives}.
\bjtitle{IEEE Transactions on Power Systems}
\bvolume{36}(\bissue{5}),
\bfpage{3865}--\blpage{3877}
(\byear{2021})
\doiurl{10.1109/TPWRS.2021.3055730}
\end{barticle}
\endbibitem

\bibitem[\protect\citeauthoryear{Pozo et~al.}{2014}]{Pozo2014Unit}
\begin{barticle}
\bauthor{\bsnm{Pozo}, \binits{D.}},
\bauthor{\bsnm{Contreras}, \binits{J.}},
\bauthor{\bsnm{Sauma}, \binits{E.E.}}:
\batitle{Unit commitment with ideal and generic energy storage units}.
\bjtitle{IEEE Transactions on Power Systems}
\bvolume{29}(\bissue{6}),
\bfpage{2974}--\blpage{2984}
(\byear{2014})
\doiurl{10.1109/TPWRS.2014.2313513}
\end{barticle}
\endbibitem

\bibitem[\protect\citeauthoryear{Garifi et~al.}{2019}]{garifi2018control}
\begin{botherref}
\oauthor{\bsnm{Garifi}, \binits{K.}},
\oauthor{\bsnm{Baker}, \binits{K.}},
\oauthor{\bsnm{Christensen}, \binits{D.}},
\oauthor{\bsnm{Touri}, \binits{B.}}:
Control of Energy Storage in Home Energy Management Systems: Non-Simultaneous
  Charging and Discharging Guarantees
(2019).
\doiurl{10.48550/arXiv.1805.00100}
\end{botherref}
\endbibitem

\bibitem[\protect\citeauthoryear{}{}]{GitHub}
\begin{botherref}
Online appendix (code).
\url{{https://github.com/eleaprat/storage-market}}
\end{botherref}
\endbibitem

\bibitem[\protect\citeauthoryear{Pereira and Pinto}{1991}]{pereira1991multi}
\begin{barticle}
\bauthor{\bsnm{Pereira}, \binits{M.V.}},
\bauthor{\bsnm{Pinto}, \binits{L.M.}}:
\batitle{Multi-stage stochastic optimization applied to energy planning}.
\bjtitle{Mathematical Programming}
\bvolume{52}(\bissue{1}),
\bfpage{359}--\blpage{375}
(\byear{1991})
\doiurl{10.1007/BF01582895}
\end{barticle}
\endbibitem

\end{thebibliography}

\end{document}